\documentclass{TPmod}
\usepackage{url}
\usepackage{setspace}
\usepackage{scrextend}

\usepackage{amsthm}
\usepackage{amssymb}
\usepackage[capitalize]{cleveref}
\usepackage{mathtools}

\usepackage{stmaryrd}

\DeclareMathOperator{\Ext}{Ext}
\def\CH{\mathrm{CH}}

\def\etaX{\eta}
\def\bc{\mathrm{bc}}
\def\ls{\mathrm{ls}}

\newcommand{\Sym}{\mathrm{Sym}}
\newcommand{\Cob}{\mathrm{Cob}}

\def\BC{\mathrm{BC}}
\def\LS{\mathrm{LS}}
\def\bimod{\text{-mod-}}
\def\lmod{\text{-mod}}
\def\rmod{\text{mod-}}
\def\yon{\mathcal{Y}}

\def\fm{\mathfrak{m}}

\newcommand{\scrM}{\EuScript{M}}
\def\scrP{\EuScript{P}}
\def\scrQ{\EuScript{Q}}
\def\scrR{\EuScript{R}}

\newcommand{\scrF}{\EuScript{F}}
\newcommand{\scrC}{\EuScript{C}}

\def\scrW{\EuScript{W}}
\def\scrHom{\mathcal{H}om}
\def\scrV{\EuScript{V}}

\def\bA{\mathbb{A}}
\def\bB{\mathbb{B}}

\def\bZ{\mathbb{Z}}

\def\bR{\mathbb{R}}
\def\bL{\mathbb{L}}

\def\bV{\mathbb{V}}
\def\cU{\mathcal{U}}
\def\cC{\mathcal{C}}
\def\cZ{\mathcal{Z}}
\def\cK{\mathcal{K}}

\def\scrD{\EuScript{D}}

\def\val{\mathrm{val}}

\def\unob{\mathrm{unob}}

\renewcommand{\rk}{\operatorname{rk}}

\newcommand{\power}[1]{\left\llbracket #1 \right\rrbracket}

\newcommand{\laurents}[1]{\left(\!\left( #1 \right)\!\right)}

\newtheorem{cons}[thm]{Construction}

\begin{document}

\title{Rational equivalence and Lagrangian tori on K3 surfaces}

\author[Sheridan and Smith]{Nick Sheridan and Ivan Smith}

\address{Nick Sheridan, School of Mathematics, University of Edinburgh, Edinburgh EH9 3FD, U.K.}
\address{Ivan Smith, Centre for Mathematical Sciences, University of Cambridge, Wilberforce Road, Cambridge CB3 0WB, U.K.}

\begin{abstract} {\sc Abstract:}  Fix a symplectic K3 surface $X$ homologically mirror to an algebraic K3 surface $Y$  by an equivalence taking a graded Lagrangian torus $L\subset X$ to the skyscraper sheaf of a point $y\in Y$.  We  show there are Lagrangian tori with vanishing Maslov class in $X$ whose class in the Grothendieck group of the Fukaya category is not  generated by Lagrangian spheres.  This is mirror to a statement about the ``Beauville--Voisin subring'' in the Chow groups of $Y$, and fits into a conjectural relationship between Lagrangian cobordism and rational equivalence of algebraic cycles. \end{abstract}

\maketitle

\section{Introduction}

\subsection{Context} Let $X=(X,\omega)$ be a symplectic Calabi--Yau manifold which is homologically mirror to a smooth algebraic  variety  $Y$ over the Novikov field $\Lambda = \C\laurents{q^\R}$, in the sense that there is a quasi-equivalence 
\[ \scrF(X)^{perf} \simeq \EuD(Y)\]
between the split-closed derived Fukaya category of $X$ and a DG enhancement of the derived category of coherent sheaves on $Y$. 
Then there is an induced isomorphism of Grothendieck groups 
\begin{equation}
\label{eqn:K0iso}
 K\left(\scrF(X)^{perf}\right) \,  \simeq \, K(\EuD(Y)),
\end{equation}
and the right-hand side is isomorphic to the algebraic $K$-theory $K(Y)$. 
This can be studied via the following twisted version of the Chern character (introduced by Mukai):
\begin{align}
\label{eqn:v}
v^\CH:K(Y) & \to \CH_*(Y)_\Q\\
\nonumber v^\CH (E) & \coloneqq \mathrm{ch}(E) \sqrt{\mathrm{td}(Y)},
\end{align}
which becomes an isomorphism after tensoring with $\Q$. 

On the other hand, Biran--Cornea construct a natural map 
\begin{equation}
\label{eqn:etaX}
\etaX: \mathrm{Cob}^\unob(X) \longrightarrow K\left(\scrF(X)^{perf}\right), 
\end{equation}
from the unobstructed Lagrangian cobordism group of $X$ to the Grothendieck group of the Fukaya category \cite{Biran-Cornea-2} (its image is the subgroup $K(\scrF(X)^{tw})$). There are numerous open questions concerning the map  $\etaX$, in particular whether it is injective.   
By composing it with the isomorphism \eqref{eqn:K0iso} and the map $v^\CH$, these questions can be related to ones about rational equivalence of algebraic cycles on $Y$.

For the rest of this paper we assume that $X$ and $Y$ are K3 surfaces, so the denominators can be removed from the map $v^\CH$, which defines an isomorphism $K(Y) \simeq \CH_*(Y)$ (see \cite[Chapter 12, Corollary 1.5]{Huybrechts:K3book}). 
Thus we have maps
\[ \Cob^\unob(X) \xrightarrow{\etaX} K\left(\scrF(X)^{perf}\right) \simeq \CH_*(Y),\]
which are our main focus.

\begin{rmk}
Relations between Lagrangian cobordism and the Grothendieck group of the mirror were first explored for elliptic curves in \cite{Haug};  the relation to rational equivalence of cycles was made explicit,  more generally for symplectic manifolds $X(B)= T^*_{\R}B/T^*_{\Z}B$ associated to tropical affine manifolds $B$, in \cite{SS:tropical}.  In that setting, cylindrical cobordisms in $\C^*\times X(B)$ were closely related to tropical curves in $\R\times B$. In this paper we again focus on Lagrangian tori, and intuition comes from the SYZ viewpoint, but we use different techniques -- partly because we do not require our tori to be fibres of a global SYZ fibration, and partly because even if such a fibration existed it would necessarily have singularities. 
\end{rmk}

\subsection{Spherical objects}
\label{Subsec:spherical_objects}

Let $Y$ be an algebraic $K3$ surface over an algebraically closed field $\BbK$ of characteristic zero (such as the Novikov field). 
Whereas $\CH_2(Y) = \Z\cdot[Y]$ and $\CH_1(Y) \simeq \Pic(Y)$ have finite rank, the behaviour of $\CH_0(Y)$ is more mysterious: in particular, when $\BbK$ is uncountable Mumford \cite{Mumford}  proved that $\CH_0(Y)$ is infinite-dimensional. 
Nevertheless, Beauville and Voisin \cite{BV} showed that any point $y$ lying on a rational curve in $Y$ has the same class $[y] = c_Y \in \CH_0(Y)$, and considered the finite-rank subgroup
\[ R(Y)\coloneqq \Z \cdot c_Y \oplus \CH_1(Y) \oplus \CH_2(Y) \subset \CH_*(Y).\]

The corresponding subgroup $R(Y) \subset K(\EuD(Y))$ can be given a purely categorical interpretation. 
Namely, let $S(Y) \subset K(\EuD(Y))$ be the subgroup generated by spherical objects; Huybrechts \cite{Huybrechts:Chow} and Voisin \cite{Voisin:0-cycles} have shown that $S(Y) \subset R(Y)$, and in fact $R(Y)$ is the saturation of $S(Y)$ (see Appendix \ref{app:BV}). 

We introduce the corresponding subgroup $S(X) \subset K\left(\scrF(X)^{perf}\right)$ generated by spherical objects, and its saturation $R(X)$. 
When $X$ and $Y$ are homologically mirror, the resulting isomorphism of Grothendieck groups \eqref{eqn:K0iso} clearly identifies the subgroups $S(X)$ and $S(Y)$, and therefore their saturations $R(X)$ and $R(Y)$.
In particular, since $R(Y)$ has finite rank, so do $R(X)$ and $S(X)$. 
It follows that the subgroup of $K(\scrF(X)^{perf})$ generated by Lagrangian spheres also has finite rank, since it is obviously contained in $S(X)$ (indeed in some cases the two are known to be equal \cite{SS:k3paper}).

\begin{rmk}
Beauville and Voisin proved that $R(Y)$ is closed under the intersection product, hence is a subring, and it is usually referred to as the \emph{Beauville--Voisin ring}.  
However we will treat it as a subgroup of $K(\EuD(Y))$, rather than as a subring. 
That is because the ring structure on $K(\EuD(Y))$ comes from the tensor product structure on $\EuD(Y)$, which we do not want to use (e.g., because existing proofs of homological mirror symmetry for K3 surfaces are not known to respect such a structure). 
Indeed, if we were allowing ourselves to use the ring structure, it would be simpler to characterise $R(Y)$ as the subring generated by $S(Y)$, rather than as the saturation of $S(Y)$.
\end{rmk}

\begin{rmk}
Some of the references quoted in this section assume that $\BbK = \C$, but this assumption can be removed using the Lefschetz principle, since rational curves and spherical objects on K3 surfaces in characteristic zero are rigid (cf. \cite{Huybrechts:K3book}).\end{rmk}

\subsection{Point-like objects}

Next we consider $0$-cycles on $Y$, in particular closed points of $Y$ (which correspond to point-like objects of the derived category, the next-simplest objects after spherical objects). 
For any subset $H \subset \CH_0(Y)$, we consider the locus $L_H(Y) \subset Y$ consisting of closed points $y \in Y$ such that $[y] \in H$. 
The following result is due to Mumford:

\begin{lem}[Mumford \cite{Mumford}]
\label{lem:mumford}
Suppose that the field of definition $\BbK$ of $Y$ is uncountable. 
If $H \subset \CH_0(Y)$ is countable, then the set  $L_H(Y)$ is a countable union of subvarieties of dimension $\leq 1$.
\end{lem}

\begin{proof}
Given a class $h \in K(Y)$, Mumford \cite{Mumford} proved that the locus $L_{\{h\}}(Y)$ is  a countable union of algebraic subsets; he also showed that $\CH_0(Y)$ is ``infinite-dimensional'', and in particular has rank $>1$, so none of these subvarieties can be full-dimensional. 
One then takes the union over all $h \in H$.
\end{proof}

We use this observation to prove the following:

\begin{main} \label{thm:non-spherical}
Let $X$ be a symplectic K3 surface and $T \subset X$ a Maslov-zero Lagrangian torus. 
Suppose there is an algebraic K3 surface $Y$ over $\Lambda$ and a closed point $y \in Y$, such that there exists a homological mirror equivalence
\[\scrF(X)^{perf} \simeq \EuD(Y)\]
taking $T$ to the skyscraper sheaf $\mathcal{O}_y$. 
Then for any countable subset $H \subset K\left(\scrF(X)^{perf}\right)$, any open neighbourhood of $T$ in $X$ contains a Maslov-zero Lagrangian torus $T'$ with $[T'] \notin H$.
\end{main}

We note that there are examples in which the hypotheses of Theorem \ref{thm:non-spherical} are satisfied.  The mirror quartic and mirror double plane are particular K3 surfaces arising from the Greene--Plesser mirror construction, for which we have:

\begin{thm}[\cite{SS:k3paper}] \label{thm:quartic_plane} If $X$ is the mirror quartic or mirror double plane, equipped with an ambient-irrational K\"ahler form $\omega$, and $T \subset X$ a Maslov-zero Lagrangian torus, then $(X,\omega)$ is mirror to an algebraic K3 surface $Y$ over $\Lambda$ by an equivalence taking $T$ to a skyscraper sheaf $\mathcal{O}_y$ over a closed point $y \in Y$.
\end{thm}

\begin{proof}
\cite{SS:k3paper} shows that $(X,\omega)$ is homologically mirror to an algebraic K3 surface $Y$ with $\Pic(Y) \simeq \langle 4 \rangle$ (for the mirror quartic) or $\langle 2\rangle$ (for the mirror double plane). 
It follows from the proof of \cite[Lemma 4.10]{SS:k3paper} that the mirror equivalence can be composed with an autoequivalence so that it takes $T$ to a skyscraper sheaf $\mathcal{O}_y$.  
\end{proof}

\begin{rmk}
One ``geometrically meaningful'' example of a countable subset of $K\left(\scrF(X)^{perf}\right)$ is the subgroup generated by Lagrangian spheres; another is its saturation. Both of these have finite rank by the previous section, and are therefore countable. For the examples of Theorem \ref{thm:quartic_plane}, whilst  $(X,\omega)$ contains Lagrangian spheres, the fact that $[T']\in K(\scrF(X)^{perf})$ is not generated by Lagrangian spheres for \emph{any} torus $T'$ near $T$ actually holds for homological reasons (see Remark \ref{rmk:notspanned}). 
However the fact that no non-zero multiple of $[T']$ is generated by Lagrangian spheres seems to have no elementary proof.
\end{rmk}

Combining Theorem \ref{thm:non-spherical} with the existence of the homomorphism $\etaX$, we immediately obtain:

\begin{cor}
\label{cor:cob-non-spherical}
In the situation of Theorem \ref{thm:non-spherical}, suppose we are furthermore given a countable set of Lagrangians $\{L_i\}$. Then any neighbourhood of $T$ contains a Maslov-zero Lagrangian torus $T'$ which is not contained in the subgroup of $\Cob^\unob(X)$ generated by the $L_i$. 
\end{cor}

However this also holds for elementary reasons of flux, see Corollary \ref{cor:infgen}. We remark that the proof via flux exhibits intriguing parallels with Mumford's argument proving infinite-dimensionality of $\CH_0(Y)$, cf. Lemma \ref{lem:infdim}. In view of such analogies, Theorem \ref{thm:non-spherical} provides modest evidence for the expectation that the map $\etaX:\Cob^\unob(X) \to K(\scrF(X)^{tw})$ is a (rational)  isomorphism when $X$ has projective mirror.

\begin{rmk}
All of our arguments crucially use the uncountability of $\R$, and in particular the impossibility of covering a positive-dimensional real vector space with countably many hyperplanes. 
The real numbers enter the story via the Novikov field $\Lambda = \C\laurents{q^\R}$. 
However we recall that if $[\omega] \in H^2(X;\Q)$ is rational, then it is possible to define the \emph{rational} Fukaya category $\scrF_{rat}(X)$ over $\Q\laurents{q^\Q}$ by restricting the objects to be ``rational'' Lagrangians, and to prove versions of homological mirror symmetry over this field (see \cite{Seidel:HMSquartic}). 
Our proof of Theorem \ref{thm:non-spherical} does not work in this case, and nor does our elementary proof of Corollary \ref{cor:cob-non-spherical} (fundamentally because you \emph{can} cover a finite-dimensional rational vector space with countably many hyperplanes). 
The coefficient field $\Q\laurents{q^\Q}$ is uncountable, so one still expects the Grothendieck group of the Fukaya category to be ``large'' by mirror symmetry, but the large number of objects of the Fukaya category really has to do with the large number of local systems on a single Lagrangian rather than a large number of geometrically distinct rational Lagrangians.   
Thus it is interesting to consider whether the cobordism group of rational Lagrangians (without local systems) should be very large, in analogy with $\CH_0(Y)$ for a K3 surface over an uncountable field, or very small, in analogy with the Bloch--Beilinson conjecture which says that $\CH_0(Y) \simeq \Z$ for a $K3$ surface over $\overline{\Q}$.
\end{rmk}

\subsection{Comments on proof of Theorem \ref{thm:non-spherical}}

The idea of the proof of Theorem \ref{thm:non-spherical} is, roughly, to show that homological mirror symmetry matches tori $T'$ near $T$ with skyscraper sheaves $\mathcal{O}_{y'}$ for points $y'$ near $y$. 
Then it suffices to show that some nearby torus $T'$ gets matched with a point $y'$ which misses $L_H(Y)$ (we identify $H$ with a subgroup of $K(Y)$ via the isomorphism $K\left(\scrF(X)^{perf}\right) \simeq K(Y)$ induced by mirror symmetry).

It is already clear that we will need to use some notion of topology on $Y$: the relevant notion is not the Zariski topology, but rather the analytic topology arising from the non-Archimedean valuation on the Novikov field, $\val: \Lambda \to \R \cup \{\infty\}$. 
We take the opportunity to introduce the notation $|\cdot| = \exp(-\val(\cdot))$ for the corresponding norm, $U_{\Lambda} \coloneqq \val^{-1}(0) = |\cdot|^{-1}(1)$ for the unitary group of the Novikov field, and $\Lambda_{>0} \coloneqq \val^{-1}(\R_{>0})$ for the subring of elements with positive valuation.

We recall how rigid analytic geometry arises on the symplectic side, following \cite{Fukaya:cyclic}. 
Lagrangian perturbations of $T$, modulo Hamiltonian isotopy, are parametrized by a neighbourhood $A$ of $0 \in H^1(T;\R)$. 
Perturbations equipped with a $U_\Lambda$-local system are parametrized by $\LS(A) = A \times H^1(T;U_\Lambda)$. 
We denote the object of $\scrF(X)$ corresponding to $a \in \LS(A)$ by $T^\ls_a$.
Assuming for simplicity that $A$ is a polytope, $\LS(A)$ can be naturally endowed with the structure of an affinoid subdomain $\val^{-1}(A) \subset H^1(T;\Lambda^*)$. 

We consider the locus $\LS_H(A) \subset \LS(A)$ of all points $a$ such that $[T^\ls_a] \in H$.
Suppose for the moment that there were an analytic map $\psi: \LS(A) \to Y$ identifying $\LS(A)$ with an analytic neighbourhood of $y$, such that the homological mirror functor identified $T^\ls_a$ with $\mathcal{O}_{\psi(a)}$.
Then $\LS_H(A) = \psi^{-1}(L_H(Y))$ would be contained in a countable union of analytic curves in $\LS(A)$ by Lemma \ref{lem:mumford}. 
The projection of an analytic curve under the ``tropicalization map'' $\val: \LS(A) \to A$ is a tropical curve; the complement of a countable union of tropical curves in $A$ is everywhere dense, and in particular non-empty. 
A point in this complement corresponds to a Lagrangian torus $T'$ such that $[(T',\xi)]$ is not generated by objects with $K$-class in $H \subset K\left(\scrF(X)^{perf}\right)$ for any unitary local system $\xi$. 
This holds in particular for the trivial local system, which would complete the proof of Theorem \ref{thm:non-spherical} assuming the existence of the map $\psi$.

If homological mirror symmetry had been proved via family Floer theory \cite{Fukaya:cyclic,Abouzaid:ICM}, and the torus $T$ were a fibre of the SYZ fibration, then the map $\psi$ would exist by construction. 
However the existing proofs of homological mirror symmetry for K3 surfaces, at least in \cite{Seidel:HMSquartic,SS}, are different: they identify the categories via deformation theory (and even if one had proved homological mirror symmetry via family Floer theory, one may want to consider tori $T$ which are not fibres of the SYZ fibration).  In particular, the fact that the homological mirror functor matches $T$ with $\mathcal{O}_y$ does not immediately imply that nearby tori get matched with skyscraper sheaves of nearby points: this requires an argument involving the deformation theory of these objects. 

It is natural to study the deformation theory of objects of the Fukaya category using the framework of bounding cochains, rather than local systems. 
We consider the set $\BC = H^1(T;\Lambda_{>0})$ of positive-energy bounding cochains on $T$, whose points $a \in \BC$ parametrize objects $T^\bc_a$ of $\scrF(X)$.  Using formal deformation theory, together with Artin's approximation theorem \cite{Artin} to guarantee a finite radius of convergence, we prove that there exists an analytic map $\varphi$ from an analytic neighbourhood $\BC' \subset \BC$ of the zero bounding cochain to an analytic neighbourhood of $y$, having the property that the homological mirror functor takes $T^\bc_a$ to $\mathcal{O}_{\varphi(a)}$. 

We recall that the object $T^\bc_a$, corresponding to the torus $T$ equipped with a bounding cochain $a \in \BC$, is quasi-isomorphic in $\scrF(X)$ to the object $T^\ls_{\exp(a)}$, corresponding to the torus $T$ equipped with the local system $\exp(a)$ \cite{Fukaya:cyclic}. 
The exponential map identifies $\BC'$ with an analytic neighbourhood $\LS'(A) \subset \LS(A)$ of the trivial local system on $T$, so we obtain an analytic map $\psi = \varphi \circ \exp^{-1}$ from $\LS'(A)$ to an analytic neighbourhood of $y$ such that the homological mirror functor takes $T^\ls_a$ to $\mathcal{O}_{\psi(a)}$. 

This is still not quite enough to implement the proof of Theorem \ref{thm:non-spherical} outlined above, because the analytic neighbourhood $\LS'(A) \subset \LS(A)$ is not `large enough': its points correspond to local systems on the single torus $T$ which are sufficiently close to the trivial local system in the analytic topology, so cannot give us information about the other tori $T'$.
A further argument is needed: we prove that the subset $\LS_H(A) \subset \LS(A)$ is a countable union of differences of analytic sets, without reference to mirror symmetry. 
This relies on a general result that we can parametrize the quasi-isomorphism classes of objects with $K$-class in $H$ by a countable set of affine varieties, together with a careful application of Chevalley's theorem. 
The intersection $\LS_H(A) \cap \LS'(A)$ coincides with $\psi^{-1}(L_H(Y))$, which is contained in a countable union of curves. 
It follows that no component of $\LS_H(A)$ is full-dimensional, which implies that this set is contained in a countable union of curves, which allows us to conclude.

\subsection{Plan}

We summarize well-known background on the formal deformation theory of objects in $A_\infty$ categories in Section \ref{Sec:def}, and on constructions of algebraic families of objects in Section \ref{Sec:ob_fam}. 
In Section \ref{Sec:tori} we recall the construction of the affinoid domains $\BC$ and $\LS(A)$ and the relationship between them, and prove that $\LS_H(A) \subset \LS(A)$ is a countable union of differences of analytic sets. 
In Section \ref{Sec:K-theory} we show how to construct the map $\varphi$ relating points on $Y$ to tori with local systems in $X$. 
The proof of Theorem \ref{thm:non-spherical} is completed in Section \ref{Sec:proof}. 
We explain the alternative proof of Corollary \ref{cor:cob-non-spherical} via flux in Section \ref{Sec:flux}, together with some interesting analogies between flux arguments and Mumford's proof of infinite-dimensionality of $\CH_0$. 
We conclude with the speculative Section \ref{Sec:OG}, concerning the mirror to the O'Grady filtration \cite{oGrady}, which is a filtration $S_0(Y) \subset S_1(Y) \subset \ldots \subset \CH_0(Y)$ whose first part is $S_0(Y) = \Z \cdot c_Y$.

\paragraph{Acknowledgements.}  Denis Auroux pointed out the elementary proof of Corollary \ref{cor:cob-non-spherical} via flux considerations. Paul Seidel suggested a different route to proving Theorem \ref{thm:non-spherical}; we found that technically harder to implement, but it influenced our understanding. Thanks to both of them, to Daniel Huybrechts for helpful conversations, and to the anonymous referee for useful queries. 

N.S.~was partially supported by a Royal Society University Research Fellowship, a Sloan Research Fellowship, and by the National Science Foundation through Grant number DMS-1310604 and under agreement number DMS-1128155. I.S. was partially supported by a Fellowship from EPSRC and, whilst holding a Research Professorship at MSRI during Spring semester 2018,  by the National Science Foundation under Grant No. DMS-1440140. 

\section{Deformations of objects in $A_\infty$ categories}\label{Sec:def}

The purpose of this section is to summarize some basic results about deformation theory of objects in $A_\infty$ categories, following e.g. \cite{FO3,Fukaya2003,ELO:I,ELO:II,ELO:III}. 
We prove only the results we will need, and make no attempt at generality.

Let $\scrC$ be an $A_\infty$ category defined over a field $\BbK$. 
We use the conventions of \cite{Seidel:FCPLT}, in which the cochain-level morphism spaces are denoted $\hom^\bullet(C_0,C_1)$ and the cohomological morphism spaces are $\Hom^\bullet(C_0,C_1) \coloneqq H^\bullet(\hom^\bullet(C_0,C_1),\mu^1)$. 

\subsection{Maurer--Cartan equation}

Let $C$ be an object of $\scrC$. 
An element $\delta \in \hom^1(C,C)$ satisfies the \emph{Maurer--Cartan equation} if
\begin{equation}
\label{eqn:MC}
\sum \mu^*(\delta,\ldots,\delta) = 0.
\end{equation}
The left-hand side is an infinite sum, so needs a reason to converge. 
We assume the morphism spaces of $\scrC$ come equipped with exhaustive complete decreasing filtrations $F^{\ge *} \hom^\bullet(C_0,C_1)$, respected by the $A_\infty$ structure maps in the obvious way, and we only consider the Maurer--Cartan equation for $\delta \in F^{\ge 1} \hom^1(C,C)$. 
The sum then converges by completeness. 

One can define a new $A_\infty$ category whose objects are pairs $(C,\delta)$ and whose $A_\infty$ structure maps are deformed by the Maurer--Cartan solutions.
For example, the formula for the deformed differential is 
\begin{align}
\nonumber
\mu^1_{\delta_0,\delta_1}: \hom^\bullet((C_0,\delta_0),(C_1,\delta_1)) &\to \hom^{\bullet+1}((C_0,\delta_0),(C_1,\delta_1)) \\
\label{eqn:defmu1}
\mu^1_{\delta_0,\delta_1}(c) &= \sum \mu^*(\delta_0,\ldots,\delta_0,c,\delta_1,\ldots,\delta_1).
\end{align}
Once again, the infinite sum converges by completeness of the filtration. 
See, e.g., \cite[Equation (3.20)]{Seidel:FCPLT} for the general formula. 

\subsection{Obstructions}\label{subsec:obs}

Let $C$ be an object of $\scrC$, and $F^{\ge *}\hom^\bullet(C,C)$ a decreasing filtration which is degreewise finite (i.e., $F^{\ge *} \hom^\bullet(C,C)$ is a finite filtration for each $\bullet$).
A family of deformations of $C$ parametrized by an affine variety $\bA$ is an element $\delta \in F^{\ge 1} \hom^1(C,C) \otimes \mathcal{O}(\bA)$ satisfying the Maurer--Cartan equation \eqref{eqn:MC} (which converges because the filtration on $\hom^\bullet(C,C) \otimes \mathcal{O}(\bA)$ is degreewise finite and in particular complete in the graded sense). 
In particular, $\delta(a) \in \hom^1(C,C)$ is a Maurer--Cartan solution for any $a \in \bA$. 
We denote the corresponding deformed object by $C_a \coloneqq (C,\delta(a))$.

For any Zariski tangent vector $v \in T_a \bA$, we have an element $v(\delta) \in \hom^1(C_a,C_a)$. 
Taking the derivative of the Maurer--Cartan equation along $v$ shows that $v(\delta)$ is $\mu^1_{\delta(a),\delta(a)}$-closed. 

\begin{defn}
The \emph{obstruction map} of the family $\delta$ at the point $a \in \bA$ is
\begin{align*}
o: T_a \bA &\to \Hom^1(C_a,C_a) \\
o(v)& \coloneqq[v(\delta)].
\end{align*}
\end{defn}

Now suppose we have families $\delta_i$ of deformations of objects $C_i$ parametrized by the same affine variety $\bA$, and a family of morphisms $f \in \hom^\bullet(C_0,C_1) \otimes \mathcal{O}(\bA)$ which is closed in the sense that 
\begin{equation}
\label{eqn:fclos}
\mu^1_{\delta_0(a),\delta_1(a)}(f) = 0
\end{equation}
in $\hom^\bullet(C_0,C_1) \otimes \mathcal{O}(\bA)$.

\begin{lem} 
\label{lem:obs}
For any $v \in T_a\bA$, we have 
\[ \mu^2_{\delta_0(a),\delta_0(a)}(o_0(v),f) + \mu^2_{\delta_1(a),\delta_1(a)}(f,o_1(v)) = 0\]
in $\Hom^\bullet(C_{0,a},C_{1,a})$, where $o_i$ is the obstruction map of $\delta_i$.
\end{lem}
\begin{proof}
Differentiating \eqref{eqn:fclos} along $v$, we find that the sum of these two terms is equal to $-\mu^1_{\delta_0(a),\delta_1(a)}(v(f))$ on the cochain level, and in particular is exact.
\end{proof}

We can also consider formal deformations of an object, which can be defined without reference to an auxiliary filtration. 
Let $(R,\fm)$ be a complete local ring. 
A family of deformations of $C$ parametrized by $\spec(R)$ is a solution of the Maurer--Cartan equation $\delta \in \hom^1(C,C) \hat{\otimes} \fm$, where the hat denotes the $\fm$-adic completion of the tensor product. 
The completeness of the $\fm$-adic filtration guarantees convergence of the Maurer--Cartan equation. 
The obstruction map is defined at the unique maximal ideal:
\begin{equation}
\label{eqn:obsform}
 o: (\fm/\fm^2)^* \to \Hom^1(C,C),
 \end{equation}
and the analogue of Lemma \ref{lem:obs} holds.

\subsection{Versality}
\label{subsec:vers}

Let $R=\BbK\power{x_1,\ldots,x_j}$ be a formal power series ring, and $\fm \subset R$ the maximal ideal. 
We say that a family $\delta$ of deformations of $C$ parametrized by $\spec(R)$ is \emph{versal} if the obstruction map \eqref{eqn:obsform} is an isomorphism. 

\begin{rmk}
The versal deformation space of an object in an $A_\infty$ category may in general be some more complicated formal scheme, but we will only consider point-like objects whose versal deformation space takes this particularly simple form.
\end{rmk}

Note that if $S=\BbK\power{y_1,\ldots,y_k}$ is another power series ring, and $\eta: R \to S$ a homomorphism, we obtain a pullback family $\eta(\delta)$ of deformations parametrized by $\spec(S)$. 
The word `versal' suggests that any family of deformations of $C$ should be pulled back from the family $\delta$. 
We now prove a result of this flavour, which we will use later. 

Let $D$ be another object of $\scrC$, quasi-isomorphic to $C$, and consider a family of deformations of $D$ parametrized by $\spec(S)$ where $S = \BbK\power{y_1,\ldots,y_k}$. 
This corresponds to a solution of the Maurer--Cartan equation $\epsilon \in \hom^1(D,D) \hat{\otimes} \mathfrak{n}$, where $\mathfrak{n} \subset S$ is the maximal ideal. 

\begin{lem}
\label{lem:versdef}
In the above situation, suppose that $\delta$ is versal. 
Then there exists a homomorphism $\eta: R \to S$, and a morphism $f \in \hom^0(C,D) \hat{\otimes} S$ which is $\mu^1_{\eta(\delta),\epsilon}$-closed, such that $f(0) \in \hom^0(C,D)$ is a quasi-isomorphism. 

If $\epsilon$ is also versal, then $\eta$ is necessarily an isomorphism.
\end{lem}
\begin{proof}
We expand in multi-indices $m$:
\[
\delta(x) = \sum_{m} \delta_{m} \,x^{m}, \qquad 
\eta(x_i) = \sum_{m} \eta^i_{m}\, y^{m}, \qquad
f(y) = \sum_m f_m \, y^m,
\]
and solve the equation
\[ \mu^1_{\eta(\delta),\epsilon}(f) = 0\] 
for $\eta$ and $f$, order-by-order in $|m|$. 
The first step of the induction is to set $f_0 \in \hom^0(C,D)$ equal to a quasi-isomorphism between $C$ and $D$. 
This is $\mu^1$-closed by definition, so the equation is solved for $|m| = 0$.
Having solved to a certain order $|m|<k$, the equation at next order yields
\begin{equation} \label{eqn:iterative}
\mu^2\left(\sum_i \delta_i \, \eta^i_m,f_0\right) + \mu^1(f_m) = (\text{known})_m \qquad \text{for all $ |m|=k$.}
\end{equation}
The RHS is $\mu^1$-closed, using the fact that $\mu^1_{\eta(\delta),\epsilon}\circ \mu^1_{\eta(\delta),\epsilon}(f) = 0$ for any $\eta$ and $f$.  
In addition, since $f_0$ is a quasi-isomorphism, composition with $f_0$ is a quasi-isomorphism of chain complexes. 
Finally, any class in $\Hom^1(C,C)$ can be represented by a linear combination of the $\delta_i$, because $\delta$ is versal. 
It follows that \eqref{eqn:iterative} admits a solution for all $m$, which completes the inductive step.

For the second part, observe that we have a diagram
\[\xymatrix{\left(\fm/\fm^2\right)^* \ar[rr]^-{[\eta]^*} \ar[d]^-{o_\delta} && \left(\mathfrak{n}/\mathfrak{n}^2\right)^* \ar[d]^-{o_\epsilon} \\
\Hom^1(C,C) \ar[r]^-{\bullet [f_0]} & \Hom^1(C,D) & \Hom^1(D,D) \ar[l]_-{[f_0]\bullet}
}\]
which commutes up to sign by Lemma \ref{lem:obs}. 
If both $\delta$ and $\epsilon$ are versal, then both obstruction maps are isomorphisms; and since $f_0$ is a quasi-isomorphism, left- or right-composition with $[f_0]$ is an isomorphism. 
It follows that $[\eta]^*$ is an isomorphism, and hence that $\eta$ is an isomorphism by the inverse function theorem.
\end{proof}

\section{Algebraic families of objects}\label{Sec:ob_fam}

Let $\scrC$ be a proper triangulated $A_\infty$ category defined over a field $\BbK$. 
The purpose of this section is to show, using as little technology as possible, how certain families of quasi-isomorphism classes of objects of $\scrC$ can be parametrized by affine varieties (we follow  \cite[Lemma A.14]{Seidel:Fukaya_lefschetz_II1/2}).
We will assume that $\scrC$ is strictly proper, i.e. finite-dimensional on the cochain level, and that it admits strict units, since this involves no loss of generality (one can replace $\scrC$ with a strictly unital minimal model) and can be arranged in appropriate models of the Fukaya category. 

\begin{rmk}
\label{rmk:param_prelim}
We are using ``parametrizing'' in a loose sense here: we mean that to each point in one of our affine varieties there corresponds an object of $\scrC$, and every object of $\scrC$ of the relevant flavour is quasi-isomorphic to that represented by a point in one of these affine varieties (but not necessarily a unique point). 
\end{rmk}

\subsection{Twisted complexes}

Recall that an object $C$ of $\scrC$ is said to be \emph{generated} by the objects $\{C_i\}$ if it is contained in the smallest triangulated subcategory of $\scrC$ which contains the objects $C_i$. 

\begin{cons}
\label{cons:param_gen}
Let $\{C_i\}$ be a countable set of objects of $\scrC$. 
We construct a countable set of affine varieties $\bA^{tw}_j$ parametrizing all quasi-isomorphism classes of objects generated by the $C_i$. 
\end{cons}

We recall the definition of the $A_\infty$ category $\scrC^{tw}$ of twisted complexes over an $A_\infty$ category $\scrC$, see \cite[Chapter 3]{Seidel:FCPLT}. 
An object is generated by the $C_i$ if and only if it is quasi-isomorphic to a twisted complex built from the $C_i$. 

Let $\scrV \subset \scrC$ be the subcategory whose objects are shifts of the $C_i$.   
We introduce a category $\scrV^{\oplus,f}$ whose objects are pairs $(V,f)$ where $V = \oplus_{j \in J} V_j$ is a finite formal direct sum of objects of $\scrV$, and $f:J \to \Z$ is a function. 
The morphisms and structure maps are as in the additive enlargement $\scrV^\oplus$, and we have a finite filtration
\[ F^{\ge *} \hom^\bullet((V,f),(W,g)) \coloneqq \bigoplus_{g(k) \ge * + f(j)} \hom^\bullet(V_j,W_k).\] 
Then we have a vector space
\[ \bA^{pre}(V,f) \coloneqq F^{\ge 1} \hom^1((V,f),(V,f))\]
parametrizing possible differentials $\delta$ in a pre-twisted complex $(V,\delta)$ respecting the filtration induced by $f$ (see \cite[Section 3l]{Seidel:FCPLT}). 
This vector space is finite-dimensional, by our assumption that $\scrC$ is strictly proper.

The subset of twisted complexes consists of those $\delta$ satisfying the Maurer--Cartan equation \eqref{eqn:MC}, which is polynomial in $\delta$ because the filtration $F^{\ge *}$ is finite. 
Thus the subset $\bA^{tw}(V,f) \subset \bA^{pre}(V,f)$ of twisted complexes is naturally an affine variety.
It is clear that all twisted complexes built from the objects $C_i$ are isomorphic to some $(V,\delta)$, where $\delta \in \bA^{tw}(V,f)$.
It is also clear that there are countably many such $(V,f)$, so this completes the construction.\qed

\begin{rmk}
\label{rmk:param}
In this case every twisted complex over the $C_i$ is \emph{isomorphic} to one represented by a point in one of our affine varieties, but our future use of the word ``parametrizing'' will not always have this stronger sense.
\end{rmk}

\subsection{Idempotent summands}\label{subsec:idemp}

Suppose now that $\scrC$ is split-closed. 
Recall that an object is said to be \emph{split-generated} by the objects $\{C_i\}$ if it is contained in the smallest split-closed triangulated subcategory of $\scrC$ which contains the objects $C_i$. 

\begin{cons}
\label{cons:split_count}
Let $\{C_i\}$ be a countable set of objects of $\scrC$. 
We construct a countable set of affine varieties $\bA^\pi_j$ parametrizing all quasi-isomorphism classes of objects split-generated by the $C_i$.
\end{cons}

We recall the definition of a homotopy idempotent $\wp$ in an $A_\infty$ category $\scrC$, and its abstract image in $\scrC^{mod}$ (see \cite[Chapter 4]{Seidel:FCPLT}). 
An object is split-generated by the $C_i$ if and only if its Yoneda image is quasi-isomorphic to the abstract image of a homotopy idempotent $\wp$ of a twisted complex $(V,\delta)$ over the $C_i$. 
 
A homotopy idempotent $\wp$ of a twisted complex $(V,\delta)$ consists of elements $\wp^d \in \hom^{1-d}_{\scrV^{tw}}(V,V)$, $d \ge 1$, satisfying
\begin{equation} \label{eqn:idempotent}
\sum_r \sum_{s_1,\ldots,s_r} \mu^r_{\scrV^{tw}}\left(\wp^{s_1},\ldots,\wp^{s_r}\right) = \left\{ \begin{array}{ll}
			\wp^{d-1} & \text{$d$ even} \\
			0 & \text{$d$ odd.}
		\end{array}\right.
\end{equation}
Explicitly, the equation $r=1$ says that the endomorphism $\wp^1$ is closed, and the equation $r=2$ says that it is idempotent up to the coboundary of $\wp^2$. 
Note that $\hom^{1-d}_{\scrV^{tw}}(V,V)$ vanishes for $d$ sufficiently large (since the total rank of $\hom^\bullet_{\scrV^{tw}}(V,V)$ is finite). 
It follows that we again have a finite set of polynomial equations in the finite set of variables $\delta,\wp^d$, whose common zero-set yields an affine variety $\bA^\pi_j$.\qed

\begin{rmk}
\label{rmk:param_pi}
Note that, as in Remark \ref{rmk:param}, every homotopy idempotent of a twisted complex over the $C_i$ is in fact \emph{isomorphic} to one represented by a point in one of the affine varieties $\bA^\pi_j$.
\end{rmk}

Let us explain how homotopy idempotents fit into the Maurer--Cartan framework of the previous section. 
Let $e_C$ denote the strict unit of an object $C\in \scrC$. Following \cite[Remark 4.10]{Seidel:FCPLT}, we consider a category $\Sigma^-\scrC$ of infinite twisted complexes $C = (C^0, C^{-1}, C^{-2}, \ldots)$ with $C^i \in \Ob\,\scrC$, and with an element $a\in \hom^k_{\Sigma^-\scrC}(C,D)$ comprising an essentially lower-triangular matrix
\[
a = (a^{ij}), \ \  a^{ij} \in \hom_{\scrC}^{k+j-i}(C^i, D^j), \quad \exists \, N \ \textrm{such \ that} \  a^{ij} = 0 \ \textrm{for} \ j<i-N.
\]
There is an exhaustive complete filtration on the morphism spaces of $\Sigma^-\scrC$, with $F^{\ge *}\hom^\bullet_{\Sigma^-\scrC}(C,D)$ consisting of morphisms with $a^{ij} = 0$ for $j < i+*$. 
There is then a category $\scrC^-$, whose objects are pairs $(C,\delta)$ with $C\in \Ob\,\Sigma^-\scrC$ and $\delta \in F^{\ge 1}\hom^1_{\Sigma^-\scrC}(C,C)$ satisfying the (resultingly convergent) Maurer--Cartan equation \eqref{eqn:MC}.

There is a canonical embedding $\scrC \hookrightarrow \scrC^-$ realised by taking the $A_{\infty}$-subcategory of objects $(C^0, 0, 0, \ldots)$. 
Furthermore, any homotopy idempotent $\wp$ for an object $C \in \scrC$ defines an object $(C,\wp)$ of $\scrC^-$ with underlying sequence $(C,C,C,\ldots)$ and 
\[
\delta_\wp^{ij} = \begin{cases} \wp^1 & i=2k, \, j=i+1\leq 0 \\ \wp^1-e_Y & i=2k+1, \, j=i+1\leq 0 \\ \wp^{j-i} & i+1<j\leq 0 \\ 0 & \textrm{else;} \end{cases}
\]
The homotopy idempotent equation \eqref{eqn:idempotent} is equivalent to the Maurer--Cartan equation for $\delta_{\wp}$.
We denote the subcategory of $\scrC^-$ consisting of objects of these two types by $\scrC^\pi$. 
The category $\scrC^{perf} \coloneqq (\scrC^{tw})^\pi$ is a model for the split-closure of $\scrC$.

Now a point $a \in \bA_j^\pi$ determines a twisted complex $(V,\delta_a)$ equipped with a homotopy idempotent $\wp_a$, and hence an object $V_a$ of $\scrC^{perf}$. 
The Maurer--Cartan elements $(\delta_a,\delta_{\wp_a})$ combine to form a single Maurer--Cartan element for an element of $(\scrC^\oplus)^-$.  
This places us within the framework of Section \ref{subsec:obs}, and in particular we have an obstruction map 
\[ o: T_a\bA \to \Hom^1_{\scrC^{perf}}(V_a,V_a).\]

\begin{rmk}
It is worth saying a word about the auxiliary filtration which ensures convergence of the Maurer--Cartan equation in this context. 
Because we assume $\scrC$ to be strictly proper, each of its $\hom$-spaces is concentrated in finitely many degrees, and it follows that the filtration $F^{\ge *}$ on the morphism spaces of $\scrC^{perf}$ is degreewise finite. 
One can combine the degreewise-finite filtration $F^{\ge *}$ with the finite filtration $G^{\ge *}$ arising from the filtration on the twisted complex, to form a degreewise-finite filtration $\mathcal{F}^{\ge *} \coloneqq \oplus_{Mi + Nj \ge *} F^{\ge i} \cap G^{\ge j}$ on the endomorphism algebra of $V_a$. 
The $A_\infty$ structure maps respect this filtration so long as $M$ and $N$ are non-negative, and by taking them to be sufficiently large we may ensure that $(\delta_a,\delta_{\wp_a}) \in \mathcal{F}^{\ge 1}\hom^1_{\scrC^{perf}}(V_a,V_a)$ (using $\delta_a \in G^{\ge 1} \cap F^{\ge 0}$, $\delta_{\wp_a} \in F^{\ge 1}$, and the finiteness of $G^{\ge *}$).
\end{rmk}

\subsection{Point-like objects}

For the purposes of this paper (in which all categories considered are 2-Calabi--Yau), an object $C$ of $\scrC$ is called \emph{point-like} if its cohomological endomorphism algebra is isomorphic to the cohomology algebra of a 2-torus, $H^\bullet(T^2;\Lambda)$ (i.e., an exterior algebra on two generators in degree $1$).

\begin{cons}
\label{cons:param_pt}
Let $\scrC$ be split-closed and admit a countable set of split-generators $\{C_i\}$. 
We construct a countable set of affine varieties $\bA^{pt}_j$ parametrizing all quasi-isomorphism classes of point-like objects in $\scrC$.
\end{cons}

For each of the affine varieties $\bA^\pi_j$ from Construction \ref{cons:split_count}, we define $\bA^{pt}_j \subset \bA^\pi_j$ to be the subset of points corresponding to point-like objects. 
It suffices to show that each $\bA^{pt}_j$ carries the structure of an affine variety. 

There is a coherent sheaf $\mathcal{E}^i \to \bA^\pi_j$ with fibre over $P=(V,f,\delta,\wp)$ being the degree $i$ part of the cohomological endomorphism algebra 
\[
H^i(\hom^*(P,P)) \coloneqq [\wp^1]\cdot H^i(\hom^*_{\scrV^{tw}}((V,\delta),(V,\delta)))\cdot [\wp^1].
\]   
The stratification by rank is lower-semicontinuous, so the subset defined by $\rk \mathcal{E}^i = \rk H^i(T^2;\Lambda)$ is a difference of affine subvarieties.  
Within this subset, the condition that the product $H^1 \otimes H^1 \to H^2$ is non-vanishing defines a Zariski-open subset, which is precisely the set of points in $\bA^\pi_j$ corresponding to point-like objects. 
A Zariski-open subset of a difference of affine subvarieties carries the structure of an affine variety, so we are done.\qed

\subsection{Objects representing classes in $K$-theory}

The next construction will make use of the following theorem of Thomason \cite{Thomason1997}, or more precisely its immediate corollary:

\begin{thm}[Thomason]
\label{thm:thomason}
Let $\scrC$ be a split-closed triangulated category. 
There is a bijection between the set of split-generating triangulated subcategories $\scrD \subset \scrC$ and the set of subgroups $G \subset K(\scrC)$. 
The subgroup corresponding to $\scrD \subset \scrC$ is $K(\scrD) \subset K(\scrC)$, and the subcategory corresponding to $G \subset K(\scrC)$ consists of all objects $C$ with $[C] \in G$.
\end{thm}

\begin{cor}
\label{cor:thomason}
Let $\scrC$ be a split-closed triangulated category, $\{S_i\}$ a set of split-generators for $\scrC$, and $\{C_j\}$ a set of objects of $\scrC$. 
If $C$ is an object of $\scrC$ such that the class $[C] \in K(\scrC)$ lies in the abelian subgroup generated by the classes $[C_j]$, then $C$ is generated by the objects $\{C_j\} \cup \{S_i \oplus S_i[1]\}$.
\end{cor}
\begin{proof}
The triangulated subcategory generated by the objects $\{C_j\} \cup \{S_i \oplus S_i[1]\}$ is split-generating, since it split-generates the objects $S_i$. 
Therefore it coincides with the subcategory of all objects $C$ with $[C]$ lying in the subgroup generated by the classes $[C_j]$ together with the classes $[S_i \oplus S_i[1]] = 0$, by Theorem \ref{thm:thomason}.
\end{proof}

We now make the following:

\begin{cons}
\label{cons:param_kclass}
Let $\scrC$ be split-closed, $\{S_i\}$ be a countable set of split-generators, and $H \in K(\scrC)$. 
Then we construct a countable set of affine varieties $\bA^H_j$ parametrizing all quasi-isomorphism classes of objects $V$ with $[V] = H$.
\end{cons}

Let $C_0$ be an object with $[C_0] = H$. 
Then any object $V$ with $[V] = H$ is generated by the objects $\{C_0\} \cup \{S_i \oplus S_i[1]\}$ by Corollary \ref{cor:thomason}. 
The set of quasi-isomorphism classes of such objects is parametrized by the countable union of affine varieties from Construction \ref{cons:param_gen}. 
The $K$-class is constant on each of these affine varieties, so a subset of this countable set of affine varieties parametrizes quasi-isomorphism classes of objects with $K$-class $H$.\qed

\section{Families of tori}\label{Sec:tori}

Let $X$ be a symplectic $K3$ surface\footnote{All Lagrangians we consider are implicitly equipped with a brane structure, which consists of an orientation, grading, spin structure and $\omega$-tame almost-complex structure $J$ such that the Lagrangian bounds no non-constant $J$-holomorphic disc and intersects no non-constant $J$-holomorphic sphere. Any orientable Lagrangian with vanishing Maslov class admits such a brane structure: in particular, the condition on $J$ is generic for virtual dimension reasons (see \cite[Lemma 8.4]{Seidel:HMSquartic}).} such that $\scrF(X)$ is non-degenerate, which implies that there is a finite set of Lagrangians $\{L_i\}$ which split-generate $\scrF(X)^{perf}$. 
Let $T \subset X$ be a Maslov-zero Lagrangian torus.
In this section we consider ``the family of objects near $T$''.

\subsection{Analytic functions}

We start by recalling some basic notions from non-Archimedean analysis, namely the algebras of convergent functions on polydiscs and polytope domains. 
Everything we present here is completely elementary and suffices for our purposes, but the interested reader can find more information about these algebras and the role they play in rigid analytic geometry in \cite{BGR,Einsiedler2006}.

Recall that the field $\Lambda$ has a non-archimedean valuation
\[
\val: \Lambda \to \bR\cup\{\infty\}, \quad \val\left(\sum a_{\lambda} q^{\lambda}\right) = \mathrm{min}\{\lambda \in \bR \, | \, a_{\lambda} \neq 0\}, \quad \val(0) = +\infty
\] with associated norm $|\cdot| : \Lambda \to \R$ given by 
\[
|x| = e^{-\val(x)}.
\]
The norm defines the \emph{analytic topology} on $\Lambda$.

A \emph{polydisc} is a subset $P(b) \subset \Lambda^n$ of the form $\{x: \val(x_i) \ge b_i, i=1,\ldots,n\}$. 
A power series 
\[
f(x) \, = \, \sum_{\nu} f_{\nu} \cdot x^{\nu} \in \Lambda \power{x_1,\ldots,x_n}
\]
\emph{converges} on the polydisc if 
\[
\lim_{\|\nu\|\to\infty} \val(f_{\nu}) + \nu \cdot b = \infty
\]
where $\|\cdot\|$ denotes any choice of norm on $\bZ^n$. 
The convergent power series form a subalgebra of the formal power series (the Tate algebra), and a convergent power series can be evaluated at any point on the polydisc. 
We will call a function from the polydisc to the Novikov field \emph{analytic} if it is given by evaluation of a convergent power series.

Next we consider the map
\begin{align*}
\val: (\Lambda^*)^n &\to \R^n,\\
\val(x_1,\ldots,x_n) &= (\val(x_1),\ldots,\val(x_n)).
\end{align*}
If $P \subset \R^n$ is a bounded rational polytope, we define the polytope domain $Y_P \coloneqq \val^{-1}(P)$. 
A formal Laurent series $f(x) = \sum_{\nu \in \Z^n} f_{\nu} \cdot x^\nu$  converges on $Y_P$ if 
\[\lim_{\|\nu\| \to \infty} \val(f_\nu) + \nu \cdot p = +\infty \quad \text{for all $p \in P$}.\]
The convergent Laurent series form an algebra, and can be evaluated at any point on the polytope domain. 
We will call a function from the polytope domain to the Novikov field analytic if it is given by evaluation of a convergent Laurent series.

A subset of a polydisc or polytope domain is called an \emph{analytic subvariety} if it is the zero-set of a finite collection of analytic functions.  
We will call a subset \emph{naively constructible} if it is a countable union of differences of analytic subvarieties. 

\subsection{Bounding cochains}

We recall that the morphism spaces of the Fukaya category come equipped with the exhaustive and complete energy filtration, so the Maurer--Cartan equation \eqref{eqn:MC} converges for positive-energy $\delta$; and a positive-energy solution of the Maurer--Cartan equation is called a bounding cochain \cite{FO3}. 
We consider the space of bounding cochains on $T$. 
Since $\hom^{\bullet}_{\scrF(X)}(T,T) \simeq C^\bullet(T;\Lambda)$ is formal, the Maurer--Cartan equation is trivial. 
Therefore the space of bounding cochains is 
\[
\BC =  H^1(T;\Lambda_{>0}).
\]
For any $a \in \BC$ we denote the corresponding object  of $\scrF(X)$ by $T^\bc_a$.
Note that for any $\varepsilon>0$, the subset $\BC_\varepsilon = H^1(T;\Lambda_{\ge \varepsilon})$ of bounding cochains with energy $\ge \varepsilon$ is a polydisc.

We will consider the subcategory of the Fukaya category $\scrF(X)$ with the finite set of objects $\{L_i\}$ together with the objects $\{T^\bc_a\}_{a \in \BC_{\varepsilon}}$. 
It is essentially immediate from the definition that the $A_\infty$ structure maps in this category are analytic in $a \in \BC_\varepsilon$, for any $\varepsilon >0$.

\subsection{Local systems}

An open neighbourhood $U$ of $T\subset X$ is fibred by Lagrangian tori over an open neighbourhood of the origin $A \subset H^1(T;\R)$, with $T$ corresponding to the fibre $F_{0}$ over the origin. (It may be helpful to imagine the case in which $X$ carries a global singular SYZ fibration and $A$ is an open subset of the base.)  

Note that we have an isomorphism
\begin{align*}
\Lambda^* & \to \R \times U_\Lambda   \\
a & \mapsto \left(\val(a),u_a\right),
\end{align*}
where $u_a = q^{-\val(a)} \cdot a$. 
This induces an isomorphism 
\[ H^1(T;\Lambda^*) \simeq 
H^1(T;\bR) \times H^1(T;U_{\Lambda}),
\]
where the projection to the first factor is the map 
\[
\val: H^1(T;\Lambda^*) \to H^1(T;\bR)
\]
induced by the valuation map $\val: \Lambda^* \to \bR$.
We define the subset 
\[
\LS(A) \coloneqq \val^{-1}(A);\] 
the points $a \in \LS(A)$ correspond to objects $T^\ls_a$ of $\scrF(X)$, consisting of the fibre $F_{\val(a)}$ equipped with the unitary local system with holonomy $u_a \in H^1(F_{\val(a)};U_\Lambda)$.

Abouzaid \cite{Abouzaid:ICM}, following Fukaya \cite{Fukaya:cyclic}, proved the following key result.

\begin{thm}[Abouzaid, Fukaya]\label{thm:fuk_trick}
Consider the subcategory of the Fukaya category $\scrF(X)$ with the finite set of objects $\{L_i\}$ together with the objects $\{T^\ls_a\}_{a \in \LS(A)}$. 
By shrinking $A$ if necessary, and making appropriate choices of perturbations and bases in the morphism spaces, we can arrange that the $A_\infty$ structure maps are analytic in $a \in \LS(A)$.
\end{thm}

The appropriate choice of basis in the morphism spaces involves some rescaling by flux, see \cite[Lemma 3.2, Equation 3.19]{Abouzaid:ICM}.

The relation between bounding cochains and local systems has been established by Fukaya in \cite[Lemma 13.1]{Fukaya:cyclic}. 
It involves the map $\exp: \BC \to \LS(A)$ induced by the exponential map $\exp: \Lambda_{>0} \to U_\Lambda $. 

\begin{thm}[Fukaya]\label{thm:bc_ls}
For any $a \in \BC$, the object $T^\bc_a$ is quasi-isomorphic to the object $T^\ls_{\exp(a)}$ in $\scrF(X)$.
\end{thm}

\subsection{Objects with fixed $K$-class}

Now let us fix a class $H \in K\left(\scrF(X)^{perf}\right)$, and consider the subset $\LS_H(A) \subset \LS(A)$ (respectively, $\BC_H \subset \BC$) consisting of objects with $[T^\ls_a] = H$ (respectively, $[T^\bc_a] = H$).

\begin{lem}
\label{lem:kclass_an}
The subset $\LS_H(A) \subset \LS(A)$ (respectively, $\BC_H \subset \BC$) is naively constructible. 
\end{lem}

We will give the proof for local systems, as the argument for bounding cochains is identical. 
In order to prove Lemma \ref{lem:kclass_an}, we observe that the objects $T^\ls_a$ parametrized by points $a \in \LS(A)$ are point-like; combining Constructions \ref{cons:param_kclass} and \ref{cons:param_pt}, we have a countable set of affine varieties parametrizing quasi-isomorphism classes of point-like objects split-generated by the $\{L_i\}$ and having the given $K$-class. 
Therefore it suffices to prove the following:

\begin{lem} \label{lem:pointlike_LS_are_analytic}
Let $\bA$ be one of these affine varieties. 
Then the subset of points $a \in \LS(A)$ consisting of objects $T^\ls_a$ lying in one of the quasi-isomorphism classes parametrized by $\bA$, is a finite union of differences of analytic subvarieties.
\end{lem}
\begin{proof}
The points of $\bA$ parametrize pairs $(\delta,\wp)$, where $\delta$ is a Maurer--Cartan element for a fixed $V = \oplus L_{i_j}[\sigma_{i_j}]$, upper triangular with respect to a fixed filtration, and $\wp$ is a homotopy idempotent for the twisted complex $(V,\delta)$, satisfying appropriate constraints so that the object $(V,\delta,\wp)$ is point-like. 
By Theorem \ref{thm:fuk_trick}, we may assume (possibly after shrinking $A$, and making appropriate choices in the definition of the Fukaya category) that the vector spaces $\hom^0(V,T^\ls_a)$, $\hom^0(T^\ls_a,V)$ and $\hom^0(T^\ls_a,T^\ls_a)$ are independent of $a$, and the $A_\infty$ structure maps are analytic in $a$. Recall that  $\hom^\bullet(T^\ls_a, T^\ls_a) \simeq C^\bullet(F_{\val(a)};\Lambda)$ (since the endomorphisms of a rank one local system are trivial). We take a strictly proper chain-level model for the Fukaya category,  as in \cite{Seidel:HMSquartic}, in which this cochain complex is concentrated in non-negative degrees and has rank one in degree zero.  The identity cochain $e_a=e_a^{\ls}$ is  independent of the holonomy of the local system and depends on $a$ at most through rescaling by a function analytic in $a$. 

We consider the subset 
\[ Q \subset \LS(A) \times \bA \times \hom^0(V,T) \times \hom^0(T,V)\]
consisting of all $(a,(\delta,\wp),f,g)$ such that 
\begin{align*}
\mu^1_{\delta,a}(f) &= 0 \\
\mu^1_{a,\delta}(g) &= 0 \\
\mu^2_{a,\delta}(\mu^2_{a,\delta}(f,\wp^1),g)  &= e_a\quad \text{is the identity in $\hom^0(T^\ls_a,T^\ls_a)$.}
\end{align*}
The object $T^\ls_a$ is a direct summand of the object $(V,\delta,\wp)$ if and only if there exist $(f,g)$ such that $(a,\alpha,f,g)$ lies in $Q$ (using the fact that $\hom^*(T,T)$ is concentrated in non-negative degrees, so any element cohomologous to $e_a \in \hom^0(T^\ls_a,T^\ls_a)$ is $e_a$ itself). 
The rank of the degree-$0$ endomorphism algebra of $(V,\delta, \wp)$ is equal to $1$, so any non-trivial direct summand is quasi-isomorphic to $(V,\delta,\wp)$.
It follows that the desired subset of $\LS(A)$ is precisely the projection $\mathrm{pr}_1(Q)$, where $\mathrm{pr}_1$ denotes projection onto the first factor.

The defining equations of $Q$ are linear in $f$, $g$ and $\wp^1$, polynomial in $\delta$, and analytic in $a$ by Theorem \ref{thm:fuk_trick}. 
They do not involve $\wp^{\ge 2}$. 
We now observe that only finitely many analytic functions in $a$ can enter into these defining equations, as there are finitely many non-zero coefficients in each of the finitely-many defining polynomials in $\delta,f,g, \wp^1$.

Let us suppose that there are $N$ of these analytic functions. 
They define an analytic map $\phi: \LS(A) \to \Lambda^N$. 
We have a corresponding set 
\[ \tilde{Q} \subset  \Lambda^N \times \bA  \times \hom^0(V,T) \times \hom^0(T,V)\]
defined by equations which are linear in $f$, $g$ and $\wp^1$, polynomial in $\delta$, and linear in the coordinates on $\Lambda^N$ standing in for the analytic functions on $\LS(A)$, so that $Q$ is the preimage of $\tilde{Q}$ under the map $\phi \times \id \times\id\times \id$. 
It follows that $\mathrm{pr}_1(Q) = \phi^{-1}\!\left(\mathrm{pr}_1\!\left(\tilde{Q}\right)\right)$. 

Since $\tilde{Q}$ is an algebraic subvariety (and $\Lambda$ is an algebraically closed field), it follows by Chevalley's theorem that $\mathrm{pr}_1\!\left(\tilde{Q}\right)$ is a finite union of differences of algebraic subsets. 
Therefore $\mathrm{pr}_1(Q)$, being its pullback under the analytic map $\phi$, is a finite union of differences of analytic sets as required.
\end{proof}

\section{Matching points with tori}\label{Sec:K-theory}

Throughout this section we work in the setting of Theorem \ref{thm:non-spherical}:   $X$ is a symplectic $K3$ surface which is homologically mirror to a $K3$ surface $Y$ over $\Lambda$, and $T \subset X$ is a Maslov-zero Lagrangian torus which the homological mirror functor sends to the skyscraper sheaf $\mathcal{O}_y$ of a closed point $y \in Y$. 

We fix a finite set of objects $\{L_i\}$ of $\scrF(X)$ which split-generate $\scrF(X)^{perf}$, and denote the subcategory with those objects by $\scrV$. 
We denote the corresponding objects of $\EuD(Y)$ by $\{\mathcal{E}_i\}$, and denote the subcategory with those objects by $\scrW$.

\subsection{Families of points near $y$}

Recall that given the split-generators $\{L_i\}$ of $\scrF(X)^{perf}$, Construction \ref{cons:param_pt} produces a countable set of affine varieties $\bA^{pt}_j$ parametrizing all quasi-isomorphism classes of point-like objects of $\scrF(X)^{perf}$. 

\begin{lem} \label{lem:express_points_as_twisted_complexes}
There exists a Zariski-open neighbourhood $\bB$ of $y \in Y$, and a morphism of affine varieties $\chi:\bB \to \bA^{pt}_j$ for some $j$, such that $\mathcal{O}_b$ is quasi-isomorphic to the image of $ \chi(b)$ under the homological mirror functor, for all $b\in \bB$.
\end{lem}

In order to prove Lemma \ref{lem:express_points_as_twisted_complexes}, it will help to start with a preliminary result about families of modules. 

We start with some background. 
Let $\bB$ be an affine variety and $\scrV$ a DG category. 
Recall the notion of a \emph{family of (right) DG $\scrV$-modules parametrized by $\bB$}: it is an $\mathcal{O}_\bB$-linear DG functor $\scrV^{op} \otimes \mathcal{O}_\bB \to \cC$, where $\cC$ is the DG category of cochain complexes of projective  $\mathcal{O}_{\bB}$-modules. 
Note that if $\scrM$ is a family of DG $\scrV$-modules parametrized by $\bB$ and $b \in \bB$, the specialization $\scrM_b$ is an ordinary DG $\scrV$-module.
We will say that a family of DG $\scrV$-modules $\scrM$ is \emph{classified by} a morphism $f:\bB \to \bA^\pi_j$ to one of the affine varieties from Construction \ref{cons:split_count}, if $\scrM_b$ is quasi-isomorphic to the abstract image of the homotopy idempotent $\wp_{f(b)}$, for all $b \in \bB$. 
 
We will call a family of DG $\scrV$-modules \emph{strictly proper} if the corresponding functor lands inside $\cC^{fin}$, the cochain complexes which are free and finitely-generated on the cochain level. 

\begin{lem}
\label{lem:build_tw_cpx}
Suppose that $\scrV$ is a homologically smooth DG category, and $\scrM$ a strictly proper family of DG $\scrV$-modules parametrized by the affine variety $\bB$. 
Then $\scrM$ is classified by a morphism $f:\bB \to \bA^\pi_j$.
\end{lem}
\begin{proof}
All functors, modules and categories in the proof are assumed to be DG. 
We start by considering the case that $\bB = \spec \BbK$ is a point, so $\scrM$ is simply a  $\scrV$-module. 
We recall some background on categories of DG modules, cf. \cite{Seidel2008c} and \cite[Section 2]{Ganatra2013}. 

The category of left $\scrV$-modules is denoted $\scrV\lmod$, the category of right $\scrV$-modules is $\rmod\scrV$, and the category of $\scrV\text{-}\scrV$-bimodules is $\scrV\bimod\scrV$ (the notation follows \cite[Section 2]{Ganatra2013}, but note that we only consider DG modules -- these form a quasi-equivalent full subcategory of the category of $A_\infty$ modules considered in \emph{op. cit.}). 
We have left and right Yoneda functors
\[ \yon^r: \scrV \to \rmod\scrV,\quad \yon^\ell: \scrV^{op} \to \scrV\lmod,\]
a tensor product functor
\[ \scrV \lmod \otimes \rmod\scrV  \to \scrV\bimod \scrV,\]
convolution functors given by tensoring with a (bi-)module, e.g.
\[ \scrM \otimes_\scrV (-): \scrV \bimod \scrV \to \rmod \scrV\]
(cf. \cite[Section 2]{Seidel2008c} -- note that the bar resolution provides a canonical way of taking an appropriately derived tensor product, but this remains a DG rather than $A_\infty$ functor), and a quasi-isomorphism of functors
\[ \scrM \otimes_\scrV \yon^\ell(-) \xrightarrow{\sim} \scrM(-)\] 
(given by the augmentation $\scrM(Y) \otimes \hom(X,Y) \to \scrM(X)$ of the bar resolution).

Combining these pieces, we obtain a homotopy-commutative diagram of functors
\[\xymatrixcolsep{5pc}\xymatrix{ \scrV^{op} \otimes \scrV \ar[r]^-{\yon^\ell \otimes \yon^r} \ar[d]_-{\scrM \otimes \id} & \scrV\bimod\scrV  \ar[d]^-{\scrM \otimes_\scrV (-)} \\
\cC \otimes \scrV \ar[r]^-{\otimes_\BbK \circ (\id \otimes \yon^r )} & \rmod\scrV.}\]

We now take twisted complexes in this diagram. 
We also observe that there is a `tensor with a finite-dimensional cochain complex' functor 
\[ \cC^{fin} \otimes \scrV \to \scrV^{tw}\]
(however note that the functor does not exist if we replace $\cC^{fin}$ with $\cC$ due to the finiteness constraint on twisted complexes). 
Using the fact that $\scrM$ is strictly proper, we obtain another diagram of DG functors
\[\xymatrix{ \left(\scrV^{op} \otimes \scrV\right)^{tw} \ar[r] \ar[d] & \left(\scrV\bimod\scrV\right)^{tw} \ar[d] \\
\scrV^{tw} \ar[r] & \left(\rmod\scrV\right)^{tw},}\]
commutative up to homotopy. 

Now by the definition of homological smoothness of $\scrV$, there exists a twisted complex $(C,\delta) \in \left(\scrV^{op} \otimes \scrV\right)^{tw}$, together with a homotopy idempotent $\wp$, whose image in $(\scrV\bimod\scrV)^{tw}$ represents the diagonal bimodule $\scrV_\Delta$. 
Since $\scrM \otimes_\scrV  \scrV_\Delta \simeq \scrM$, it follows by commutativity that the image of this homotopy idempotent in $\scrV^{tw}$ represents $\scrM$. 

Thus, given a resolution of the diagonal, we obtain a natural way of forming a twisted complex together with a homotopy idempotent representing a given strictly proper $\scrV$-module. 
We now apply this to a strictly proper family of $\scrV$-modules $\scrM$ over an arbitrary affine variety $\bB$. 
We obtain a family of twisted complexes with homotopy idempotents representing the fibres; and this family must be pulled back from one of the affine varieties $\bA^\pi_j$ by the universality of the construction of the latter (see Remark \ref{rmk:param_pi}).
\end{proof}

\begin{proof}[Proof of Lemma \ref{lem:express_points_as_twisted_complexes}]
Since $Y$ is a smooth variety, we may assume that the split-generators $\mathcal{E}_i$ are bounded complexes of locally free sheaves of finite rank. 

To define our DG enhancement $\EuD(Y)$ of the bounded derived category $D^b Coh(Y)$, we use a \v{C}ech model following \cite[Section 5]{Seidel:HMSquartic}, 
i.e., we choose a finite affine open cover $\mathfrak{U}$ of $Y$ and set 
\[\hom^\bullet_{\EuD(Y)}(\mathcal{E},\mathcal{F}) \coloneqq \check{C}^\bullet(\mathfrak{U},\scrHom(\mathcal{E},\mathcal{F})).\]

We may assume that $y$ lies in the common intersection of the affine open sets in our cover. 
The complexes $\mathcal{E}_i^\vee$ are locally free of finite rank, so their restrictions to a sufficiently small open affine neighbourhood $\bB$ of $y$ are free. 
Shrinking $\bB$ further, we may assume that it is contained in the common intersection of the sets in our cover. 
If $\iota: \bB \hookrightarrow Y$ is the inclusion, then the complexes $\check{C}^\bullet(\mathfrak{U},\iota_*\iota^*\mathcal{E}_i^\vee)$ are free $\mathcal{O}_\bB$-modules of finite rank, whose fibres over a point $b \in \bB$ coincide with $\hom^\bullet_{\EuD(Y)}(\mathcal{E}_i,\mathcal{O}_b)$. 

Thus, we have constructed a strictly proper family of right DG $\scrW$-modules parametrized by $\bB$, representing the Yoneda images of the skyscraper sheaves $\mathcal{O}_b$ for $b \in \bB$. 
Since $\scrW$ is a homologically smooth DG category we may apply Lemma \ref{lem:build_tw_cpx}, to obtain a family of homotopy idempotents of twisted complexes over $\scrW$ parametrized by $\bB$. 
Twisted complexes and homotopy idempotents are functorial under $A_\infty$ functors (see \cite[Section 3m]{Seidel:FCPLT} for twisted complexes; homotopy idempotents are functorial because they can be regarded as non-unital $A_\infty$ functors from the ground field \cite[Section 4b]{Seidel:FCPLT}, which can be post-composed with any $A_\infty$ functor), so applying the homological mirror functor we obtain a corresponding family of homotopy idempotents of twisted complexes over $\scrV$. 
The objects represented by points $b \in \bB$ are all clearly point-like, so this family is classified by a morphism $\bB \to \bA_j^{pt}$, by the universality of the construction of the latter (see Remark \ref{rmk:param_pi}).
\end{proof}

\subsection{Versality of the family of points}

It is well-known that the family of skyscraper sheaves of points is a versal family, cf. \cite[Section 5]{ELO:III}. 
In particular we have:

\begin{lem}\label{lem:ptsvers}
The resulting obstruction map
\[ T_y \bB \xrightarrow{o} \Hom^1_{\scrF(X)^{perf}}(\chi(y),\chi(y))\]
is an isomorphism.
\end{lem}
\begin{proof}
This is a reflection of the fact that the homomorphism
\begin{equation}
\label{eqn:obspt}
 T_y Y \to \Ext^1(\mathcal{O}_y,\mathcal{O}_y),
 \end{equation}
arising from the connecting homomorphism for the short exact sequence of coherent sheaves
\begin{equation}
\label{eqn:sespt}
 0 \to T^*_y Y \otimes \mathcal{O}_y \to \mathcal{O}_Y/\mathcal{I}_y^2 \to \mathcal{O}_y \to 0,
 \end{equation}
 is an isomorphism. 
 
To make the connection we use the notion of a \emph{naive short exact sequence} of $\scrW$-modules
\[ 0 \to \scrP \to \scrQ \to \scrR \to 0,\]
which consists of the inclusion $\scrP \subset \scrQ$ of a submodule, i.e. a subspace $\scrP(W) \subset \scrQ(W)$ for each object $W$ preserved by all structure maps, together with the projection to the quotient module  $\scrR$. 
A choice of splittings of the short exact sequences of vector spaces determines a connecting morphism in $\hom^1_{\rmod\scrW}(\scrR,\scrP)$, whose cone is strictly isomorphic to $\scrQ$, cf. \cite[Section 2]{Seidel2008c} and \cite[Section 3p]{Seidel:FCPLT}. 

Let $\scrM$ denote the strictly proper family of $\scrW$-modules parametrized by $\bB$ that was constructed in the proof of Lemma \ref{lem:express_points_as_twisted_complexes}, representing the Yoneda images of skyscraper sheaves of points $b \in B$. 
Taking the tensor product of $\scrM$ with the short exact sequence of $\mathcal{O}_\bB$-modules
\begin{equation}
\label{eqn:OBses} 0 \to \left(\fm /\fm^2\right) \otimes \mathcal{O}_\bB/\fm \to \mathcal{O}_\bB/\fm^2 \to \mathcal{O}_\bB/\fm \to 0,
\end{equation}
we obtain a naive short exact sequence of $\scrW$-modules representing the Yoneda image of the short exact sequence of coherent sheaves \eqref{eqn:sespt}.

The connecting homomorphism is an element of $\Hom^1_{\rmod\scrW}(\scrM_y,\fm/\fm^2 \otimes \scrM_y)$, which can be regarded as a homomorphism
\[ T_y\bB \to \Hom^1_{\rmod\scrW}(\scrM_y,\scrM_y).\]
It is routine to verify that this coincides with the Yoneda image of the isomorphism \eqref{eqn:obspt}, and hence is an isomorphism. 
On the other hand, there is a canonical splitting for the short exact sequence \eqref{eqn:OBses} given by the map $\mathcal{O}_\bB/\fm \to \mathcal{O}_\bB/\fm^2$ sending $f \mapsto f(y)$; this determines a splitting of the short exact sequence of $\scrW$-modules, and it is routine to verify that the corresponding connecting morphism coincides on the chain level with the obstruction map for the family of modules $\scrM$. 
Hence the latter is an isomorphism, and therefore so is the obstruction map for the corresponding family of objects of $\scrW^{perf}$.
\end{proof}

\subsection{Families of bounding cochains on $T$}

The next result uses Artin's approximation theorem \cite[Theorem 1.2]{Artin}, which says roughly that if an analytic system of equations admits a formal solution, then it admits an analytic solution. 
Precisely, consider a finite collection $\{f_i(x,y)\}_{1\leq i\leq N}$ of convergent analytic functions in variables $x  \in \Lambda^n$, $y \in \Lambda^m$. 
Then we have:

\begin{thm}[Artin] \label{thm:Artin}
If there is a formal (power series) solution $q(x)$ of  $\{f_i(x,q(x)) = 0\}_{1\leq i\leq N}$, with $q(0) = 0$, then there is a convergent solution $Q(x)$. Moreover, for any $k$, there exists a convergent solution whose $k$-jet coincides with that of the formal solution.
\end{thm}

We combine Artin's theorem with the formal deformation theory from Section \ref{subsec:vers} to prove:

\begin{lem}\label{lem:local_one_dim}
There exists an analytic neighbourhood $\bB_\epsilon$ of $y \in \bB$, together with an analytic map $\eta: \bB_\epsilon \to \BC$ sending $y$ to $0$, which is an analytic isomorphism onto a neighbourhood of $0 \in \BC$, such that the object $\mathcal{O}_z$ is quasi-isomorphic to the image of $T^\bc_{\eta(z)}$ under the homological mirror functor, for all $z\in \bB_{\varepsilon}$. 
\end{lem}

In order to prove the Lemma we will work with the category $\scrQ \supset \scrV$ whose objects are the chosen finite set of split-generators $\{L_i\}$ for $\scrF(X)$, together with the objects $T^\bc_a$ for $a\in \BC$. Using homotopy units as in \cite{FO3}, there is a strictly unital chain-level model for $\scrQ$. 
We remark that we work with a slightly different construction of the Fukaya category from \cite{FO3}, following \cite{Seidel:HMSquartic}, which has the advantage that it is strictly proper (as is required for our constructions). 
Homotopy units have been constructed in this framework in \cite{Ganatra2013}.

Since $\scrQ$ admits strict units, we may consider the split-closure $\scrQ^{perf}$ constructed in Section \ref{subsec:idemp}. 
Even though $\scrQ$ is strictly proper, $\scrQ^{perf}$ need not be. 
Nevertheless, it has a weaker finite-dimensionality property. 
Namely, we recall that $\scrQ^{perf}$ has two types of objects: objects $X$ of $\scrQ^{tw}$, and objects $(Y,\wp)$ corresponding to a homotopy idempotent of an object $Y$ of $\scrQ^{tw}$. 
Using the fact that $\scrQ$ is strictly proper, one easily verifies that the morphism groups $\hom^k_{\scrQ^{perf}}(X, (Y,\wp))$ and $\hom^k_{\scrQ^{perf}}((Y, \wp), X)$ are finite-dimensional in each fixed degree $k$. 

In particular, the category $\scrQ^{perf}$ contains both the objects $T^\bc_a$ and objects quasi-representing idempotents up to homotopy as parametrized by the affine varieties $\bA_j^{pt}$, and morphisms between objects of the respective types are finite-dimensional in each degree.  (One can also construct a split-closure for $\scrQ$ as a subcategory of the category of modules $\scrQ^{mod}$, which however does not satisfy this finite-dimensionality condition; in the application below, this could lead to infinite systems of analytic equations, for which Artin's theorem would not guarantee convergent solutions with a common positive radius of convergence.) 

\begin{proof}[Proof of Lemma \ref{lem:local_one_dim}]
The proof is structurally somewhat  similar to that of Lemma \ref{lem:pointlike_LS_are_analytic}. Let $\chi: \bB \to \bA = \bA_j^{pt}$ be the map of Lemma \ref{lem:express_points_as_twisted_complexes}, given explicitly as
\[
\chi(z) = (V, \delta(z), \wp(z)).
\]
Note that $\chi(z)$ determines an object of the category $\scrQ^{perf}$. 
The vector spaces $E^i = \hom^i_{\scrQ^{perf}}(\chi(z), T^\bc_a)$ do not depend on $a$ or $z$,
and we have a map
\[
\Phi:  \BC \times \bB \times E^0 \to E^1, \qquad \Phi(a,z,r) = \mu^1_{\scrQ^{perf}}(r).
\]
Schematically,
\begin{equation}
\label{eqn:Phi}
\Phi(a,z,r) = \sum \, \pm \, \mu^\ast(a,\ldots, a, r, \delta,\ldots,\delta,\delta_\wp,\delta,\ldots,\delta,\delta_\wp,\ldots,\delta_\wp,\delta,\ldots,\delta)
\end{equation}
where $\delta$ and $\delta_\wp$ are algebraic functions of $z$.
Using the fact that $\scrQ$ is strictly proper and hence its $\hom$-spaces are degree-bounded, there is a fixed $N$ such that $r^{ij} = 0$ for all $j<i-N$, for all $r \in E^0$. 
It follows that we have an upper bound on the number of times that $\delta_\wp$ can appear in \eqref{eqn:Phi}. 
We also have an upper bound on the number of times that $\delta$ can appear between consecutive appearances of $\delta_\wp$, by the strict lower triangularity of $\delta$, and therefore we obtain an upper bound on the total number of times that $\delta$ can appear. 
Thus $\Phi(a,z,r)$ is a polynomial in $z$ whose coefficients are convergent power series in $a \in \BC_{\varepsilon'}$ for any $\varepsilon' >0$, and linear in $r$, so it is analytic on $\BC_{\varepsilon'} \times \bB \times E^0$.

Now the restriction of the family of objects $T^\bc_a$ to a formal neighbourhood of $a=0$ is clearly versal; similarly the restriction of the family of objects $\chi(z)$ to a formal neighbourhood of $z=y$ is versal by Lemma \ref{lem:ptsvers}. 
Therefore,  there exists a formal solution $(a(z),r(z))$ to $\Phi(a(z),z,r(z))=0$ by Lemma \ref{lem:versdef}, with $a(0) = 0$ and $r(0) = r_0$ equal to the quasi-isomorphism between $T= T^\bc_0$ and $\chi(y)$, and the function $z \mapsto a(z)$ is an isomorphism of formal neighbourhoods. 
Artin's theorem then yields a convergent analytic solution, whose 1-jet at the origin co-incides with that of the formal solution, and it follows by the inverse function theorem \cite[Theorem 10.10]{Abhyankar1964} that $z\mapsto a(z)$ is a local analytic isomorphism.

We have an analytic family of cochain complexes over an analytic neighbourhood of $y\in \bB$, finite-dimensional in each degree and with acyclic fibre at $z=0$, given by the mapping cones 
\[
\mathrm{Cone}\,\left( \xymatrix{ \hom^i_{\scrQ^{perf}}\left(T^\bc_{a(z)},\chi(z)\right) \ar[rrr]^{\mu^2_{\scrQ^{perf}}(\cdot , \,r(z))} &&  & \hom^i_{\scrQ}\left(T^\bc_{a(z)},T^\bc_{a(z)}\right) } \right).
\] Since the rank of cohomology is upper semicontinuous, for any $i$ there exists an analytic neighbourhood of $0$ on which the degree-$i$ cohomology vanishes. It follows that there exists an analytic neighbourhood $\bB_\varepsilon$ of the origin such that the cohomological unit lies in the image of $\mu^2_{\scrQ^{perf}}(\cdot,\, r(z))$ for each $z \in\bB_{\varepsilon} $, which implies that $T^\bc_{a(z)}$ is an idempotent summand of $\chi(z)$. However, both objects are point-like, in particular have rank one endomorphisms in degree zero so admit no non-trivial idempotent decomposition, therefore $r(z)$  is a quasi-isomorphism for $z\in\bB_{\varepsilon}$. 

Since analytic maps are continuous, the cocycle $a(z)$ belongs to $\BC$ in some analytic neighbourhood of $z=0$. Shrinking $\bB_{\varepsilon}$ if necessary,  we obtain a bounding cochain $a(z)$ on $T$ quasi-isomorphic to $\mathcal{O}_z$ for each $z\in \bB_{\varepsilon}$, and this defines the required map $\eta(z) \coloneqq a(z)$ from $\bB_{\varepsilon}$ to $ \BC$. 
\end{proof} 

\section{Proof of Theorem \ref{thm:non-spherical}}\label{Sec:proof}

\subsection{Naively constructible subsets}\label{sec:nc_sets}

Let $Y$ be a polydisc or a polytope domain over a bounded convex rational polytope. 
We will say that a naively constructible subset $V = \cup_i (D_i \setminus E_i)$ in $Y$ has codimension $\ge 1$ if we may take $D_i \neq Y$ for all $i$.

\begin{lem}
\label{lem:const_sub}
Let $Y$ be as above. If $V\subset Y$ is a naively constructible subset whose intersection with an open set (in the analytic topology) has codimension $\ge 1$, then $V$ itself has codimension $\ge 1$.
\end{lem}
\begin{proof}
Suppose to the contrary that some $D_i = Y$, with $E_i \neq Y$. 
Then $V$ would contain a set of the form $Y \setminus E_i$ where $E_i$ is a proper analytic subvariety. 
It would follow that the open set could be covered by $E_i$ together with a countable set of proper analytic subvarieties, which is impossible.
\end{proof}

\subsection{Bounding cochains}

We now consider the locus $\BC_H \subset \BC$ consisting of all $a$ such that $[T^\bc_a]$ is contained in the countable subset $H \subset K\left(\scrF(X)^{perf}\right)$.

\begin{lem}\label{lem:non_sph_bc}
The locus $\BC_H$ is naively constructible of codimension $\ge 1$.
\end{lem}

\begin{proof}
We know that $\BC_H$ is naively constructible by Lemma \ref{lem:kclass_an}. 
By Lemma \ref{lem:mumford}, the locus of points $y \in \bB^{}_\varepsilon$ such that $[y]$ is contained in the corresponding subset $H \subset \CH_*(Y)$ is naively constructible of codimension $\ge 1$. 
It follows that the intersection of $\BC_H$ with the image of the local isomorphism $\chi: \bB^{}_\varepsilon \to \BC$ provided by Lemma \ref{lem:local_one_dim} has the same property, and therefore that $\BC_H$ has the same property by Lemma \ref{lem:const_sub}.
\end{proof}

\subsection{Local systems}

We now consider the locus $\LS_H(A) \subset \LS(A)$ consisting of all $a$ such that $[T^\bc_{a}] \in H$.

\begin{lem}\label{lem:non_sph_ls}
The locus $\LS_H(A)$ is naively constructible of codimension $\ge 1$.
\end{lem}
\begin{proof}
We know that $\LS_H(A)$ is naively constructible by Lemma \ref{lem:kclass_an}. 
The exponential map defines an isomorphism from $\BC_\varepsilon$ to an open subset of $\LS(A)$, for any $\varepsilon >0$. 
It identifies $\BC_H$ with the intersection of $\LS_H(A)$ with this open subset, by Theorem \ref{thm:bc_ls}. 
Since $\BC_H$ is naively constructible of codimension $\ge 1$ by Lemma \ref{lem:non_sph_bc}, it follows that $\LS_H(A)$ is naively constructible of codimension $\ge 1$ by Lemma \ref{lem:const_sub}.
\end{proof}

\subsection{Valuation image}

Let $P$ be a bounded convex rational polytope, and $Y_P$ the corresponding polytope domain. 
Any analytic subvariety $Z \subset Y_P$ admits a decomposition into finitely many irreducible components, each of which has a well-defined (Krull) dimension (see  \cite[Section 7]{BGR} for background). 

The next result describes the image of an irreducible analytic subvariety of a polytope domain $Y_P$ under the valuation map $\val: Y_P \to P$.  
It dates back to \cite{Bieri1984} (see also \cite{Einsiedler2006}). 
A proof in the present setting follows from \cite[Corollary 6.2.2]{Berkovich2004} (see also  \cite[Proposition 5.4]{Gubler}, or \cite[Section 8]{Rabinoff} for the special case of an algebraic hypersurface in $Y_P$  in the language of rigid analytic spaces).

\begin{thm}
\label{thm:val_image}
Let $Z \subset Y_P$ be an irreducible analytic subvariety of dimension $k$.  Then the image $\val(Z) \subset P$ is a locally finite rational polyedral complex of dimension $k$.
\end{thm}

We can now complete the proof of Theorem \ref{thm:non-spherical}:

\begin{proof}[Proof of Theorem \ref{thm:non-spherical}]
By Lemma \ref{lem:non_sph_ls}, $\LS_H(A)$ is contained in a countable union of proper  analytic subvarieties, each of which has dimension $\le 1$ by Krull's principal ideal theorem. 
Therefore Theorem \ref{thm:val_image} implies that the image $\val(\LS_H(A)) \subset A$ is contained in a countable union of locally finite rational polyhedral complexes of dimension $\le 1$. 
In particular there exists a point $a \in A$ that does not lie in the image. 
This means that the torus $T' = F_a$ has the property that no unitary local system on $T'$ defines an object with $K$-class lying in $H$. 
This holds in particular for the trivial local system, which completes the proof.
\end{proof}

\section{Cobordisms and flux}
\label{Sec:flux}

In this section we describe the most basic obstructions to the existence of a cobordism, which arise from flux. 
Related ideas appear in \cite{Reznikov} and \cite[Proposition 4.6]{Mikhalkin2018}.

\subsection{Cobordism groups} 

Let $X=(X,\omega) $ be a closed symplectic manifold equipped with an $\infty$-fold Maslov cover of its oriented Lagrangian Grassmannian. 
We define a Lagrangian brane to be a Lagrangian submanifold $L \subset X$ equipped with a grading relative to this Maslov cover, and a spin structure.

A cobordism between Lagrangian branes $L, L' \subset X$ is a Lagrangian brane $V \subset (\C\times X, \omega_{\C} \oplus \omega)$ which projects properly to $\C$, and which outside a compact set co-incides with the Lagrangian branes $(L' \times (-\infty,-1)) \cup (L \times (1,\infty))$.  There is an obvious notion of Lagrangian cobordism between tuples $\bL^- = \{L_j^-\}$ and $\bL^+ = \{L_k^+\}$ by allowing the cobordism to have more ends at either or both of positive and negative real infinity, or of a Lagrangian nullcobordism if one of the sets of ends is empty.  See \cite{Biran-Cornea,Biran-Cornea-2} for background.

We let $\Cob(X)$ denote the Lagrangian (brane) cobordism group, i.e. the abelian group generated by Lagrangian branes modulo relations from brane cobordisms.  The `U-turn' cobordism shows that shifting the grading of a Lagrangian by $[1]$ changes the sign of the corresponding class in $\Cob(X)$. (The group $\Cob^{\unob}(X)$ appearing in the Introduction, in which relations are imposed only when they arise from Floer-theoretically unobstructed cobordisms, comes with a canonical map $\Cob^{\unob}(X) \to \Cob(X)$.) 

\begin{rmk} A  variant on the definition, introduced in \cite{SS:tropical}, is to consider cobordisms over $\C^*$, i.e. Lagrangian branes in $(\C^*\times X, \omega_{\C^*} \oplus \omega)$ which carry $\bL^{\pm}$ over  radial half-lines in a neighbourhood of the ends of $\C^*$. Any cobordism over $\C$ yields a cobordism over $\C^*$ by quotienting the plane $\C$ by a  large  imaginary translation.  Cylindrical cobordisms play no specific role here, but are more directly comparable to  the Chow group of 0-cycles on the mirror as explained in \cite{SS:tropical}.
\end{rmk}

Traces of Hamiltonian isotopies, and of Lagrange surgeries which respect brane structures appropriately, both yield cobordisms. 
The construction extends to Morse--Bott Lagrange surgeries, i.e. surgeries along clean intersections.

\begin{example}\label{eg:spher1}
Let $X$ be a symplectic Calabi--Yau surface which contains Lagrangian spheres $S, S'$ which meet transversely at two points $\{p,q\}$ of the same Maslov grading.  There is a shift $\ell$ and a graded Lagrange surgery $T = S[\ell] \stackrel{p+q}{\longrightarrow} S'$ which is a  Lagrangian torus graded cobordant to $S[\ell] \sqcup S'$. Varying the surgery parameters, one obtains a one-parameter family of Hamiltonian isotopy classes of such tori $T$, all of which have the same class $S[\ell] + S'$ in $\Cob(X)$.
\end{example}

\begin{example}\label{eg:spher2}
Let $X$ be a symplectic Calabi--Yau surface which contains a Lagrangian sphere $S$ which meets a Lagrangian torus $T$ cleanly in a circle $\gamma = S\cap T \subset T$. There is a one-parameter family $T_t$ of Hamiltonian isotopy classes of Lagrangian tori obtained as the images of $T$ by isotopies of flux $t\alpha$, where $\alpha \in H^1(T;\R)$ vanishes on the class of the loop $\gamma$, which one can arrange to also meet $S$ cleanly in a family of parallel circles $\gamma_t$. The Morse--Bott Lagrange surgery of $T_t$ and $S$ is a Lagrangian sphere $S_t$, whose Hamiltonian isotopy class $S'$ is independent of $t$. It follows that the one-parameter family of tori $T_t$ all have the same class $S - S'$ in $\Cob(X)$.  (This example has a tropical interpretation, considering the family of SYZ fibres lying over  interior points in an edge of the tropicalization of a rational curve.)
\end{example}

\subsection{Flux constraints}

If $L \subset X$ is a Lagrangian, we recall that the \emph{flux homomorphism}
\[ \{\text{small deformations of $L$}\}/\text{Ham. isotopy} \xrightarrow{\varphi_L} H^1(L;\R)\]
is given by integrating the symplectic form over the cylinders swept by loops in $L$ under the deformation. 
It is an isomorphism onto a neighbourhood of the origin $\cU_L \subset H^1(L;\R)$: the inverse takes a class $\alpha \in \cU_L$ to the graph of a closed one-form representing $\alpha$ in de Rham cohomology, in a Weinstein neighbourhood of $L$. 
We denote this  Hamiltonian isotopy class by $L(\alpha)\coloneqq \varphi_L^{-1}(\alpha)$.

Similarly, if $V \subset \C \times X$ is a cobordism with (positive and negative) ends $\bL = \{L_i\}$, there is a flux homomorphism $\varphi_V$ fitting into a commutative diagram
\begin{equation}
\label{eqn:fluxcom}
\xymatrix{\{\text{small deformations of $V$}\}/\text{Ham. isotopy} \ar[r]^-{\varphi_V} \ar[d] & H^1(V;\R) \ar[d] \\
\{\text{small deformations of $\bL$}\}/\text{Ham. isotopy} \ar[r]^-{\varphi_{\bL}}& H^1(\partial V;\R),}
\end{equation}
where $\varphi_{\bL}$ is the direct sum of all of the flux homomorphisms $\varphi_{L_i}$, mapping to 
\[ \bigoplus_i H^1(L_i;\R) \simeq H^1(\partial V;\R).\]
The only difference from the standard case of a single Lagrangian recalled above is in the construction of the inverse to the flux homomorphism $\varphi_V$: one chooses the Weinstein neighbourhood of $V$, and the closed one-forms representing classes in $H^1(V;\R)$, to respect the product structure near $\partial V$. 
To prove that the diagram \eqref{eqn:fluxcom} commutes, one observes that $\omega_\C$ vanishes over a neighbourhood of $\partial V$ because the projection of the neighbourhood to $\C$ is contained in a one-dimensional subset.

Now let $\bL=\{L_i\}$ be a finite tuple of Lagrangian branes in $X$, and let $\cU_\bL \coloneqq \prod_i \cU_{L_i}$. 
We consider the map
\begin{align*}
f_{\bL}: \cU_\bL &\to \Cob(X) \\
f_{\bL}(\alpha) & \coloneqq \sum_i L_i(\alpha_i).
\end{align*}
Note that each Hamiltonian isotopy class $L_i(\alpha_i)$ gives a well-defined element of $\Cob(X)$ because Hamiltonian-isotopic Lagrangians are cobordant.

\begin{lem}\label{lem:gen2flux}
Let $\bL^+,\bL^-$ be finite tuples of Lagrangians. Then 
\[ \left\{\left(\alpha^+,\alpha^-\right) \in \cU_{\bL^+} \times \cU_{\bL^-}: \ f_{\bL^+}\left(\alpha^+\right) = f_{\bL^-}\left(\alpha^-\right)\right\} = \bigcup_{\bV \in \mathcal{V}} \cZ_\bV,\]
where the indexing set $\mathcal{V}$ is countable, and to each $\bV \in \mathcal{V}$ there is associated an oriented manifold $V$ with oriented boundary $\partial V = \partial^+ V \sqcup \overline{\partial^- V}$, and an identification $\partial^\pm V \simeq \sqcup_i L_i^\pm$, such that $\cZ_\bV$ is identified with an open subset of an affine subspace parallel to $\im(H^1(V;\R) \to H^1(\partial V;\R))$.
\end{lem}
\begin{proof}
We define a \emph{topological} cobordism between $\bL^-$ and $\bL^+$ to be a smooth map $i: V \to \C \times X$ which is asymptotic to the Lagrangians outside a compact set (in the same way that a Lagrangian cobordism is); we require $V$ to be oriented, compatibly with the boundary. 
We say two topological cobordisms are homotopic if they are homotopic relative to a neighbourhood of $\partial V$. 
By choosing each $\cU_{L}$ to be simply connected, we may ensure that the set of homotopy classes $\bV$ of topological cobordisms between $\bL^-$ and $\bL^+$ can be canonically identified with those between  $\bL^-(\alpha^-)$ and $\bL^+(\alpha^+)$, for any $\alpha^\pm \in \cU_{\bL^\pm}$. 
The set $\mathcal{V}$ of such homotopy classes is countable, by the simplicial approximation theorem. 

Associated to a homotopy class of topological cobordisms is a cohomology class $[i^*(\omega_\C \oplus \omega)] \in H^2(V,\partial V;\R)$, whose non-vanishing obstructs the existence of a Lagrangian representative. 
Moving the ends $\bL^\pm$ of the cobordism by flux $\alpha^\pm$ changes this class by the image of $\alpha^\pm$ under the map $H^1(\partial V;\R) \to H^2(V,\partial V;\R)$. 
Applying the long exact sequence for the pair $(V,\partial V)$, we see that the set of fluxes $\alpha^\pm$ for which the obstruction vanishes is (the intersection of $\cU_{\bL^+} \times \cU_{\bL^-}$ with) an affine subspace $\cZ'_{[V]}$ parallel to $\im(H^1(V;\R) \to H^1(\partial V;\R))$.

Now, any relation $\sum_i L_i^+(\alpha_i^+) = \sum_i L_i^-(\alpha_i^-)$ is generated by a finite set of Lagrangian cobordisms $\{V_j\}$. 
Gluing the ends of the $V_j$ together we obtain a (possibly immersed) Lagrangian cobordism $V$ between $\bL^-(\alpha^-)$ and $\bL^+(\alpha^+)$; thus $(\alpha^+,\alpha^-)$ lies in the affine subspace $\cZ'_{[V]}$. 
Furthermore, deformations of $V$ give rise to deformations of the $V_j$ by restriction, so the cobordism relation can be deformed precisely in the directions corresponding to $\im(H^1(V;\R) \to H^1(\partial V;\R))$, by \eqref{eqn:fluxcom}. 
It follows that the subset $\cZ_{[V]} \subset \cZ'_{[V]}$ represented by Lagrangian cobordism relations is open, completing the proof.
\end{proof}

Now suppose that $\dim(X) = 2n$.  
We have an $n$-form $\Omega$ on $\cU_\bL$, given by 
\begin{equation}\label{eqn:Omega}
\Omega(\alpha_1,\ldots,\alpha_n) \coloneqq \int_{\sqcup_i L_i} \alpha_1 \cup \ldots  \cup \alpha_n
\end{equation}
for $\alpha_i \in T\cU_\bL \simeq H^1(\sqcup_i L_i;\R)$. 
When $n=2$ this is a symplectic form.

\begin{lem}\label{lem:2dflux}
The restriction of $\Omega$ to each of the subspaces $\cZ_{\mathbb{V}}$ from Lemma \ref{lem:gen2flux} vanishes. 
When $n=2$, the subspaces $\cZ_\bV$ are furthermore Lagrangian.
\end{lem}
\begin{proof}
If $\alpha_i \in T\cU_\bL \simeq H^1(\partial V;\R)$ lie in the tangent space to $\cZ_{[V]}$, then they are the image of some $\beta_i \in H^1(V;\R)$ under the restriction map $H^1(V;\R) \to H^1(\partial V;\R)$ by Lemma \ref{lem:gen2flux}. 
Thus 
\[  \Omega(\alpha_1,\ldots,\alpha_n) = \int_{\partial V} \beta_1 \cup \ldots \cup \beta_n = 0, \]
because $[\partial V]$ vanishes in $H_n(V;\R)$.

In the case $n=2$, this means each $\cZ_\bV$ is isotropic for the symplectic form $\Omega$. 
The additional fact that the image of $H^1(V;\R) \to H^1(\partial V;\R)$ is half-dimensional, and hence Lagrangian, is proved by a well-known argument using Poincar\'e--Lefschetz duality.
\end{proof}

\subsection{Infinite-dimensionality}

We use the results of the previous section to analyse the fibres of the maps $f_\bL : \cU_\bL \to \Cob(X)$.

\begin{lem}
\label{lem:isot}
Let $\bL$ be a finite tuple of Lagrangian branes. 
Then for any $z \in \Cob(X)$ we have
\[ f_\bL^{-1}(z) = \bigcup_\bV \cK_\bV\]
where the union is countable, and each $\cK_\bV$ is an open subset of an affine subspace on which $\Omega$ vanishes.
\end{lem}
\begin{proof}
It is immediate from Lemma \ref{lem:gen2flux} that
\[ f_\bL^{-1}(f_\bL(\alpha)) \simeq \bigcup_\bV \cK_\bV\]
where $\cK_\bV \coloneqq \cZ_\bV \cap \left(\cU_\bL \times \{\alpha\}\right) \subset \cU_\bL$. 
Furthermore, $\Omega$ vanishes on each $\cZ_\bV \subset \cU_\bL \times \cU_\bL$ by Lemma \ref{lem:2dflux}, and hence on its intersection with $\cU_\bL \times \{\alpha\}$.
\end{proof}

\begin{rmk}
Examples \ref{eg:spher1} and \ref{eg:spher2} exhibit one-dimensional families of Lagrangian tori in symplectic K3 surfaces, whose class in $\Cob(X)$ is constant. 
On the other hand, the $n=2$ case of Lemma \ref{lem:isot} shows that there is no two-dimensional family of Lagrangian tori whose class in $\Cob(X)$ is constant, because an isotropic subspace is at most half-dimensional.
\end{rmk}

For the next Corollary, we will apply Lemma \ref{lem:isot} in the case that $\bL = \{L\}$ is a singleton, so $\cU_{\bL} = \cU_L \subset H^1(L;\R)$ is a neighbourhood of $0 \in H^1(L;\R)$.   Note that if  $L$ is a torus or an oriented surface of strictly positive genus, then the $n$-form $\Omega$ defined in \eqref{eqn:Omega} is non-vanishing on $\cU_L$.

\begin{cor}
\label{cor:infgen}
Let $\{L_i\}$ be a countable set of Lagrangian branes, and $L$ a Lagrangian brane such that $\Omega$ is non-vanishing on $\cU_L$.   Then there is a deformation $L'$ of $L$ whose class in $\Cob(X)$ is not generated by the $\{L_i\}$.
\end{cor}
\begin{proof}
Note that any subspace of $\cU_L$ on which $\Omega$ vanishes must have positive codimension, and $\cU_L$ cannot be covered by a countable set of such subspaces. 
Since the set of classes in $\Cob(X)$ generated by the set $\{L_i\}$ is countable, the result follows immediately from Lemma \ref{lem:isot}.
\end{proof}

Corollary \ref{cor:infgen} shows that, if $X$ contains a Lagrangian brane $L$ as in the statement, then $\Cob(X)$ is not countably generated. 
When $X$ is four dimensional, this simply means that $L$ contains a Lagrangian of genus $\ge 1$. 
In fact the following result shows that in this case, $\Cob(X)$ is not even `finite-dimensional':

\begin{lem}
\label{lem:infdim}
Suppose $X$ is four dimensional, and contains a Lagrangian brane $L$ of genus $\ge 1$. Then $\Cob(X)$ cannot be covered by the images of maps $f_{\bL_i}$, for any countable set of finite tuples $\{\bL_i\}$ such that the dimensions of the $\cU_{\bL_i}$ are bounded above.
\end{lem}
\begin{proof}
Suppose to the contrary that $\Cob(X)$ is covered by the images of such maps $f_{\bL_i}$; then in particular, $\im(f_{\bL^+})$ is covered by these images for any $\bL^+$. 
This implies that the projections of the sets $\cup_{\bV_i} \cZ_{\bV_i} \subset \cU_{\bL^+} \times \cU_{\bL_i}$ from Lemma \ref{lem:gen2flux} to $\cU_{\bL^+}$ cover all of $\cU_{\bL^+}$. 
Since there are countably many $\cZ_{\bV_i}$, and they are Lagrangian and in particular half-dimensional by Lemma \ref{lem:2dflux}, this can only be true if $\dim(\cU_{\bL^+}) \le \dim(\cU_{\bL_i})$ for some $i$. 
However the dimension of $\cU_{\bL^+}$ can be made arbitrarily large by taking sufficiently many copies of $L$, so we have a contradiction to the bound on the dimension of $\cU_{\bL_i}$.
\end{proof}

\begin{rmk}
In the loose analogy between cobordism and Chow groups, 
Lemma \ref{lem:infdim} is mirror to Mumford's theorem \cite{Mumford} that $\CH_0(S)$ is infinite dimensional for a K3 surface $S$ over $\Lambda$. 
Our proof bears striking similarities with Mumford's, both obtaining a dimension bound from the isotropic condition.
\end{rmk}

\begin{rmk}
Observe that there are countably many Hamiltonian isotopy classes of Lagrangian spheres in any symplectic manifold of dimension $\ge 4$. Thus Corollary \ref{cor:infgen} shows that for any Lagrangian $L$ as in the statement, there exists a deformation $L'$ of $L$ which is not generated by Lagrangian spheres.  
In fact, Corollary \ref{cor:infgen} can be strengthened to show there exists a deformation $L'$, no non-zero multiple of which is generated by the $L_i$; so there exists a deformation $L'$, no non-zero multiple of which is generated by Lagrangian spheres. 
We mention this slight strengthening because of the relationship with the categorical characterisation of the Beauville--Voisin subring as the saturation of the subgroup generated by spherical objects, cf. Section \ref{Subsec:spherical_objects}.
\end{rmk}

\section{The O'Grady filtration}\label{Sec:OG}

Let $Y$ be an algebraic K3 surface over $\C$.  There is a filtration
\[
\bZ \cdot c_Y = S_0(Y) \subset S_1(Y) \subset \cdots \subset S_g(Y) \subset \cdots \subset \CH_0(Y),
\]
introduced by O'Grady \cite{oGrady} and further studied in \cite{Voisin:0-cycles}, with lowest-order part generated by the Beauville--Voisin class $c_Y$ (see Section \ref{Subsec:spherical_objects}).   In this final section we discuss a possible symplectic counterpart. Recall from the Introduction that $\bZ\cdot c_Y \subset S(Y) \subset K(Y)$ lies in the saturation of the group generated by spherical objects, so one expects the mirror to $S_0(Y)$ to involve Lagrangians (a multiple of) whose cobordism class is generated by Lagrangian spheres.

\subsection{Configurations in symmetric products}

Let $Y$ be an algebraic K3 surface over $\C$.  By definition, O'Grady sets 
\[
S_g(Y) \coloneqq \left \{ z \in \CH_0(Y) \  \big | \ z = z' + a c_Y, \  \textrm{with $z'$ effective  of degree $g$ and $a\in \Z$} \right\}.
\] 
This has a number of equivalent and more geometric characterisations; for instance O'Grady proves that $S_g(Y)$ comprises the cycles $z \in \CH_0(Y)$ for which there is some curve $C\subset Y$ with normalisation $C^+$, with the sum of genera of components of $C^+$ being at most $g$, and with  $z \in \mathrm{Im}(\CH_0(C) \to \CH_0(Y))$ (see \cite[Corollary 1.7]{oGrady} and \cite[Proposition 2.7]{Voisin:0-cycles}).  

Voisin \cite{Voisin:0-cycles} has given another characterisation of the induced finite filtration
\[
S_0^k(Y) \subset S_1^k(Y) \subset S_2^k(Y) \subset \cdots \subset S_k^k(Y) \subset \, \CH_0(Y) \cap \mathrm{deg}^{-1}(k),
\]
where $\deg: \CH_0(Y) \to \Z$ is the degree homomorphism and $k\geq 0$. Let
\[
Z_k(Y, z) \coloneqq \left \{(x_1,\ldots,x_k) \in \Sym^k(Y) \ \big| \   \sum\, [x_i] = z \in \CH_0(Y) \,\right \}
\]
be the orbit of the $0$-cycle $z$ under rational equivalence;  this is a union of algebraic subsets, and its dimension is by definition the supremum of the dimensions of its irreducible components. 
Mumford \cite{Mumford} showed the components of $Z_k(Y,z)$ to be isotropic for the holomorphic symplectic form induced on (the smooth locus of) the symmetric product by the holomorphic symplectic form on $Y$, hence of dimension $\le k$.  

\begin{thm}[Voisin]  \label{thm:isotropic_Bside} For $g\leq k$, 
\[
S_g^k(Y) = \left\{z \in CH_0(Y) \cap \deg^{-1}(k) \ \big | \ Z_k(Y,z) \neq \emptyset, \ \dim(Z_k(Y,z)) \geq k-g \, \right\}.
\]
In particular, $Z_k(Y,z)$ is Lagrangian if and only if  $z=k\cdot c_Y$. \end{thm}

The dimension constraint $\dim (Z_k(Y,z)) \le k$, arising from the isotropic condition, is analogous to Lemma \ref{lem:isot}. 
Thus it is natural to wonder if Theorem \ref{thm:isotropic_Bside} has a similar analogue.

\subsection{Target genus of a cobordism}

Fix a symplectic K3 surface $X$ containing a Maslov-zero Lagrangian torus $L$, and let $\cU \coloneqq \cU_L$.
For any $k \in \Z_{>0}$, let $\bL^k \coloneqq \{\text{$k$ copies of $L$}\}$, so $\cU_{\bL^k} = \cU^k$. 
For $z \in \Cob(X)$ we define
\[
Z_k(X,z) \, \coloneqq \,  f_{\bL^k}^{-1}(z),
\]
which is an isotropic subset of $\cU^k$ and therefore has dimension $\le k$, by Lemma \ref{lem:isot}.

\begin{rmk} 
In view of Theorem \ref{thm:isotropic_Bside}, and the fact that $\bZ\cdot c_Y$ lies in the saturation of the subgroup generated by spherical objects, it is natural to ask if $\dim(Z_k(X,z)) = k$ if and only if some multiple of $z$ is generated by Lagrangian spheres. 
Note that Examples \ref{eg:spher1} and \ref{eg:spher2} give one-parameter families of tori with constant class in $\Cob(X)$, so $\dim(Z_1(X,z)) = 1$; in both of these examples the tori are generated by Lagrangian spheres by construction.
\end{rmk}

Let $\Cob^k(X,\cU) \subset \Cob(X)$ denote the image of $f_{\bL^k}$. We define a filtration by minimal possible ``target genus'' of a cobordism with given input: 
\[
\Cob^k(X,\cU)_g \coloneqq \left \{ k\textrm{-tuples which are cobordant to a union of Lagrangians of total genus} \ \leq g  \right\}.
\]
Then 
\[
\Cob^k(X,\cU)_0 \subset \Cob^k(X,\cU)_1 \subset \Cob^k(X,\cU)_2 \subset \cdots \subset \Cob^k(X,\cU)_k = \Cob^k(X,\cU), 
\]
and the lowest order piece $\Cob^k(X,\cU)_0$ comprises exactly the $k$-tuples of elements of $\cU$ whose sum is generated by Lagrangian spheres. 
This filtration should be compared with \cite[Speculation 0.1 (b)]{Shen2017a}.

\begin{lem}
\label{lem:OGdim}
If  $z \in \Cob^k(X,\cU)_g$, then $Z_k(X,z)$ has dimension $\geq k-g$.
\end{lem}

\begin{proof}
Let $z = \sum_i L_i$ where $L_i$ have total genus $g$. 
Setting $\bL^- = \{L_i\}$ and $\bL^+ = \bL^k$, we have
\[ Z_k(X,z) = f_{\bL^+}^{-1}\left(f_{\bL^-}(0)\right).\]
Lemma \ref{lem:2dflux} then implies that $Z_k(X,z)$ is the intersection of a countable union of subsets $\cZ_\bV \subset \cU_{\bL^-} \times \cU_{\bL^+}$ (which are Lagrangian, hence of dimension $g+k$) with the subset $\{0\} \times \cU^k$ (which has codimension $2g$). 
Therefore it has dimension at least $g+k - 2g = k-g$.
\end{proof}

Suppose now that we have a homological mirror equivalence $\scrF(X)^{perf} \simeq \scrD(Y)$ with an algebraic K3 surface $Y$ over $\Lambda$, taking $T$ to $\mathcal{O}_y$. Suppose furthermore that the O'Grady filtration is well-defined and satisfies the previously discussed properties for $Y$, even though it is defined over $\Lambda$ rather than $\C$. 
Then we expect that $\Cob^{unob,k}(X,\cU)_g$ corresponds, under mirror symmetry, to a subset of $S^k_g(Y)$.
In light of Theorem \ref{thm:isotropic_Bside}, Lemma \ref{lem:OGdim} provides some modest evidence for this expectation. 

However it is also clear that $\Cob^{unob,k}(X,\cU)_g$ needs to be enlarged in order to be mirror to all of $S^k_g(Y)$. 
For example, the pieces $S_g(Y)$ of the O'Grady filtration are invariant under multiplication by $\Z$ \cite[Corollary 1.7]{oGrady}. The filtration $\Cob^k(X,\cU)_g$ introduced above need not have this property (for example, there exist symplectic $K3$ surfaces $X$ which contain Maslov-zero Lagrangian tori, but only in primitive homology classes \cite{SS:k3paper}, in which case it is clear that $k \cdot \Cob^1_1 \not\subseteq \Cob^k_1$ for any $k >0$).  
The most naive modification would be to allow repeated copies of the same Lagrangian to only contribute once to the total genus; however the proof of Lemma \ref{lem:OGdim} would become much harder with this modified version and we will not pursue it further.

\appendix

\section{The Beauville--Voisin ring and spherical objects}
\label{app:BV}

Let $Y$ be an algebraic $K3$ surface over an algebraically closed field $\BbK$ of characteristic zero, $R(Y) \subset \CH_*(Y)$ the Beauville--Voisin subring, and $S(Y) \subset \CH_*(Y)$ the subgroup generated by Mukai vectors of spherical objects in $\EuD(Y)$. 
The aim of this appendix is to prove the following:

\begin{thm}
\label{thm:sat}
The subgroup $R(Y) \subset \CH_*(Y)$ is the saturation of the subgroup $S(Y) \subset \CH_*(Y)$.
\end{thm}

(We recall that a subgroup $A$ of an abelian group $B$ is \emph{saturated} if the quotient $B/A$ is torsion free, and the \emph{saturation} $A^{\mathrm{sat}}$ of a subgroup $A$ is the smallest saturated subgroup  containing $A$.)

The hard part of the proof is to show that $S(Y) \subset R(Y)$, which was proved by Huybrechts and Voisin \cite{Huybrechts:Chow,Voisin:0-cycles}. 
To finish the proof, we start by showing:

\begin{lem}
The subgroup $R(Y) \subset \CH_*(Y)$ is saturated.
\end{lem}
\begin{proof}
Observe that the inclusion $\Z \cdot c_Y \subset \CH_0(Y)$ is split by the degree map, and this induces a splitting of the inclusion of the Beauville--Voisin ring $R(Y) \subset \CH_*(Y)$:
\[ \CH_*(Y) \simeq R(Y) \oplus \CH_0(Y)_{\mathrm{hom}}.\]
The result now follows from the fact that $\CH_0(Y)_{\mathrm{hom}}$ is torsion free \cite{Rojtman,Bloch}.
\end{proof}

The proof of Theorem \ref{thm:sat} is now completed by the following:

\begin{lem} The saturation of $S(Y) \subset R(Y)$ is all of $R(Y)$.
\end{lem}
\begin{proof}
We equip $\CH_*(Y)$ with the Mukai pairing, which has the property that
\begin{equation}
\label{eqn:HRR}
 \chi(E,F) = - \langle v^\CH(E),v^\CH(F)\rangle
 \end{equation}
for objects $E,F$ of $\EuD(Y)$ (see \cite[Chapter 9]{Huybrechts:K3book}). 
It is clear from \eqref{eqn:HRR} that the Mukai vector $v^\CH(E)$ of a spherical object $E$ has square $-2$, and we know that it lies in $R(Y)$ by the aforementioned result of Huybrechts and Voisin. 

In fact conversely, every element of $R(Y)$ of square $-2$ is the Mukai vector of a spherical object: this follows for $\BbK = \C$ by a result of Kuleshov \cite{Kuleshov}, and thus for general $\BbK$ by the Lefschetz principle. 
In particular the subgroup $S(Y) \subset R(Y)$ is generated by the elements of square $-2$. 
Observing that $R(Y) \simeq U \oplus \Pic(Y)$ where $U$ denotes the hyperbolic lattice $\left( \begin{smallmatrix} 0&1 \\ 1&0 \end{smallmatrix}\right)$, the result now follows from Lemma \ref{lem:min2sat} below.
\end{proof}

\begin{lem}
\label{lem:min2sat}
Let $L$ be a lattice with non-zero pairing, and $S \subset U \oplus L$ the sublattice spanned by elements of square $-2$. 
Then the saturation of $S \subset U \oplus L$ is all of $U \oplus L$.
\end{lem}
\begin{proof}
For any $\ell \in L$, the class $(1,(-\ell^2-2)/2, \ell)$ has square $-2$ so lies in $S$. 
Subtracting $(1,-1,0) \in S$, we see that $(0, -\ell^2/2, \ell) \in S$. It follows that
\[ (0,-m^2\ell^2/2,m\ell) - m\cdot (0,-\ell^2/2,\ell) = (m-m^2)\cdot \ell^2/2 \cdot (0,1,0) \in S.\]
We may choose $(m-m^2)\cdot \ell^2/2 \neq 0$, so $(0,1,0) \in S^{\mathrm{sat}}$. 
Similarly $(1,0,0)\in S^{\mathrm{sat}}$. 
Subtracting appropriate multiples of these classes from $(1,(-\ell^2 - 2)/2,\ell)$ shows that $(0,0,\ell) \in S^{\mathrm{sat}}$, so $S^{\mathrm{sat}} = U \oplus L$.
\end{proof}

We conclude with an example, showing that the inclusion $S(Y) \subset R(Y)$ may be strict:

\begin{ex}\label{ex:satneeded}
If $L = \langle 2n \rangle$ where $n \not \equiv 3\,( 4)$, then the class $(1,0,0) \in U \oplus L$ does not lie in $S$.
\end{ex}
\begin{proof}
Let $(a,b,c)$ be a class of square $-2$, so $ab+nc^2 = -1$. 
If $n$ is even, this implies $a+b \equiv 0\,(2)$; if $n \equiv 1\,(4)$, this implies $a+b+c \equiv 0\,(2)$; hence the same is true of any element of $S$. 
Since $(1,0,0)$ does not satisfy either of these equations it cannot lie in $S$.
\end{proof}

\begin{rmk}
\label{rmk:notspanned}
Example \ref{ex:satneeded} shows that if $\Pic(Y) \simeq \langle 2n\rangle$ with $n \not \equiv 3(4)$, then the Beauville--Voisin class $c_Y$ is not spanned by Mukai vectors of spherical objects. 
This implies that on the mirror, the homology class of a fibre of the SYZ fibration is not spanned by Lagrangian spheres.
\end{rmk}

\bibliographystyle{amsalpha}
\bibliography{mybib}

\end{document}